\documentclass[a4paper,11pt,reqno]{amsart}

\date{July 2, 2019}

\usepackage[latin1]{inputenc}

\usepackage{
amssymb,
tikz,
xcolor,
}

\usepackage[hidelinks]{hyperref}
\usepackage{bbm}

\usepackage{geometry}
\newtheorem{thm}{Theorem}[section]
\newtheorem{lemma}[thm]{Lemma}
\newtheorem{proposition}[thm]{Proposition}
\newtheorem{corollary}[thm]{Corollary}

\theoremstyle{remark}
\newtheorem{remark}[thm]{Remark}

\newcommand{\R}{\mathbb{R}}
\newcommand{\C}{\mathbb{C}}
\newcommand{\Z}{\mathbb{Z}}

\newcommand{\cK}{\mathcal{K}}

\newcommand{\cD}{\operatorname{\mathcal{D}}}
\newcommand{\cL}{\mathcal{L}}

\renewcommand\phi{\varphi}
\newcommand{\eps}{\varepsilon}

\renewcommand{\geq}{\geqslant}
\renewcommand{\leq}{\leqslant}
\newcommand{\Sum}{\displaystyle \sum}

\DeclareMathOperator{\const}{const}

\DeclareMathOperator{\sgn}{sgn}

\newcommand{\Sph}{\mathbb{S}}

\renewcommand{\geq}{\geqslant}
\renewcommand{\leq}{\leqslant}

\newcommand\I{{\rm{i}}}

\newcommand\N{{\mathbb{N}}}
\newcommand{\be}{\begin{equation}}
\newcommand{\ee}{\end{equation}}

\title[Sharp decay estimates for critical Dirac equations]{Sharp decay estimates for critical Dirac equations}

\author[W. Borrelli]{William Borrelli}
\address[W. Borrelli]{Centro De Giorgi, Scuola Normale Superiore, Piazza dei Cavalieri 3, I-56100 , Pisa, Italy.}
\email{william.borrelli@sns.it}

\author[R. L. Frank]{Rupert L. Frank}
\address[R. L. Frank]{Mathematisches Institut, Ludwig-Maximilans Universit\"at M\"unchen, Theresienstr. 39, 80333 M\"unchen, Germany, and Mathematics 253-37, Caltech, Pasa\-de\-na, CA 91125, USA}
\email{rlfrank@caltech.edu}

\begin{document}

\begin{abstract}
We prove sharp pointwise decay estimates for critical Dirac equations on $\R^n$ with $n\geq 2$. They appear for instance in the study of critical Dirac equations on compact spin manifolds, describing blow-up profiles, and as effective equations in honeycomb structures. For the latter case, we find excited states with an explicit asymptotic behavior. Moreover, we provide some classification results both for ground states and for excited states.
\end{abstract}

\thanks{\copyright\, 2019 by the authors. This paper may be
reproduced, in its entirety, for non-commercial purposes.\\
U.S.~National Science Foundation grant DMS-1363432 (R.L.F.) is acknowledged.}

\maketitle


\section{Introduction}
\label{sec-intro}

\subsection{Main results}

This paper is devoted to the study of solutions of the nonlinear Dirac equation
\be\label{criticaldirac}
\cD\psi=\vert\psi\vert^{2^\sharp-2}\psi \qquad\mbox{on}\quad\R^{n}
\ee
with the \emph{critical exponent}
$$
2^\sharp:=\frac{2n}{n-1} \,,
$$
as well as certain extensions of this equation of the form
\be\label{criticaldiracgen}
\cD\psi=h(x,\psi)\,\psi \qquad\mbox{on}\quad\R^{n} \,,
\ee
where $h$ is a matrix-valued function which is (approximately) homogeneous of degree $2^\#-2$. We will always assume that $n\geq 2$.

As we describe below in more detail, there are at least two motivations for studying these equations, one coming from the spinorial analogue of the Yamabe problem in geometric analysis and the other one from an effective description of wave propagation in two-dimensional systems with the symmetries of a honeycomb lattice.

We are interested in two different aspects of solutions of equations \eqref{criticaldirac} and \eqref{criticaldiracgen}. The first one concerns sharp bounds on the decay of solutions. The second one concerns the classification of solutions possessing some extra symmetry. The link between these two aspects is that our classification results show that our decay estimates are always sharp for `ground state solutions' but, on the other hand, that `excited state solutions' in general exhibit a faster decay, at least if $n=2$.

We proceed to a precise description of our results. For $n\geq 2$ let $N=2^{\left[\frac{n+1}{2} \right]}$, where $[\cdot]$ denotes the integer part of a real number, and let $\alpha_1,\ldots,\alpha_n$ be $N\times N$ Hermitian matrices satisfying the anticommutation relations
\be\label{anticommutation}
\alpha_{j}\alpha_{k}+\alpha_{k}\alpha_{j}=2\delta_{j,k} \,,\qquad 1\leq j,k\leq n \,.
\ee
Such matrices exist and form a representation of the \emph{Clifford algebra} of the Euclidean space (see e.g. \cite{diracoperators}). Different choices of matrices satisfying \eqref{anticommutation} correspond to unitarily equivalent representations. For $n=2$ we will choose
\begin{equation}
\label{eq:alpha2}
\alpha_1 = \begin{pmatrix}
 0 & 1 \\ 1& 0
\end{pmatrix},
\qquad
\alpha_2 = \begin{pmatrix}
 0 & -\I \\ \I & 0
\end{pmatrix}.
\end{equation}
For $n\geq 3$ we will choose the $\alpha_j$ of a particular block-antidiagonal form, namely, let $\sigma_1,\ldots,\sigma_n$ be $\frac{N}{2}\times\frac{N}{2}$ Hermitian matrices satisfying analogous anticommutation relations as in \eqref{anticommutation}, namely,
$$\sigma_{j}\sigma_{k}+\sigma_{k}\sigma_{j}=2\delta_{j,k},\qquad 1\leq j,k\leq n. $$
Then the matrices
\be\label{sigmaanticommutation}
\alpha_{j}=\begin{pmatrix}0 & \sigma_{j} \\ \sigma_{j} & 0 \end{pmatrix}, \qquad 1\leq j\leq n,
\ee
satisfy \eqref{anticommutation} and we shall work in the following with this choice. We write $\boldsymbol{\sigma}=(\sigma_{j})^{n}_{j=1}$ and $a\cdot\boldsymbol\sigma:= \sum_{j=1}^n a_j\sigma_j$ for $a\in\R^n$. For $n=2$, we define $a\cdot\boldsymbol \sigma = a_1 + \I a_2$ for $a\in\R^2$.

Given a choice of matrices $\alpha_j$, the Dirac operator is defined as an operator acting on functions on $\R^n$ with values in $\C^N$ by
\be\label{diracdefinition}
\cD:=-\I\boldsymbol{\alpha}\cdot\nabla=-\I\sum^{n}_{j=1}\alpha_{j}\partial_{x_{j}} \,.
\ee
A more detailed presentation of Dirac operators and Clifford algebras can be found, for instance, in \cite{diracoperators,jost}. 

We assume throughout that $h$ is a function from $\R^n\times\C^N$ to $N\times N$ matrices satisfying
\begin{equation}
\label{eq:hass}
\sup_{x\in\R^n,\, z\in\C^N} |z|^{-2^\#+2} \| h(x,z) \| < \infty
\end{equation}
where $\|\cdot\|$ denotes a matrix norm on $\C^{N\times N}$. The most important case is where $h(x,z)=|z|^{2^\#-2}$, in which case \eqref{criticaldiracgen} reduces to \eqref{criticaldirac}.

We say that $\psi\in L^{2^\sharp}(\R^n,\C^N)$ is a weak solution to \eqref{criticaldiracgen} if
\be\label{weaksolution}
\int_{\R^n}\langle\cD\phi,\psi\rangle \,dx =\int_{\R^n} \langle\phi,h(x,\psi)\psi\rangle \,dx, \qquad\text{for all}\ \phi\in C^{\infty}_{c}(\R^n,\C^N).
\ee
We will give below some explicit examples of solutions of \eqref{criticaldirac}. Modifying the arguments in \cite{massless} in a straightforward way, one can show existence of solutions for a wide class of functions $h$.

Finally, for $0\neq x\in\R^n$ we introduce the following unitary matrices
\begin{equation}
\label{eq:defu}
\mathcal{U}(x) = \begin{pmatrix}
0 & -\I\frac{x}{|x|}\cdot\boldsymbol\sigma \\ \I \frac{x}{|x|}\cdot\boldsymbol\sigma & 0
\end{pmatrix}
\quad\text{if}\ n\geq 3 \,,
\qquad
\mathcal{U}(x) = \begin{pmatrix}
 0 & -\I\ \frac{x_1-\I x_2}{|x|} \\ \I\ \frac{x_1+\I x_2}{|x|} & 0
\end{pmatrix}
\quad\text{if}\ n= 2 \,.
\end{equation}

The following is the first main result of this paper.

\begin{thm}\label{thm-main}
Assume \eqref{eq:hass} and let $\psi\in L^{2^\sharp}(\R^n,\C^N)$ be a weak solution of \eqref{criticaldiracgen}. Then $\psi\in C^{0,\alpha}(\R^n,\C^N)$ for any $\alpha<1$ and there is a $\Psi\in\C^N$ such that for any $\alpha<1$ there is a $C_\alpha<\infty$ such that
$$
\left| \psi(x) - |x|^{-n+1} \mathcal{U}(x) \Psi \right| \leq C_\alpha |x|^{-n+1-\alpha}
\qquad\text{for all}\ |x|\geq 1 \,.
$$
\end{thm}

Later on, we will give natural examples showing that both $\Psi\neq 0$ and $\Psi=0$ can occur.

Note that the theorem implies that
\be\label{eq:decaymain}
\vert\psi(x)\vert\leq \frac{C}{1+\vert x\vert^{n-1}} \qquad \text{for all}\ x\in\R^n \,.
\ee
In particular,
\be\label{integrability}
\psi\in L^{p}(\R^n, \C^N) \qquad \mbox{for all}\ \frac{n}{n-1}<p\leq\infty
\ee
and
\begin{equation}
\label{eq:integrabilityweak}
\psi\in L^{\frac{n}{n-1},\infty}(\R^n ,\C^N) \,.
\end{equation}
(Here $L^{\frac{n}{n-1},\infty}(\R^n ,\C^N)$ denotes the weak Lebesgue space.) The fact that $\Psi\neq 0$ can occur shows that, in general,
\be\label{exceptional}
\psi\notin L^{\frac{n}{n-1}}(\R^n ,\C^N).
\ee
In some applications it is crucial whether solutions of \eqref{criticaldirac} are square-integrable or not. Our result shows that square-integrability holds always in $n\geq 3$ and may not hold in $n=2$. We will investigate the case $n=2$ below in more detail.

We also note that if one replaces $\mathcal D$ by the \emph{massive} Dirac operator then solutions exhibit exponential decay. This was shown in \cite{berthiergeorgescu} for $n=3$ and extended in \cite{spectralnabilecomech} to arbitrary $n$; see also \cite{cassano}. The fact that there are indeed solutions to the massive analogue of \eqref{criticaldirac} was shown in \cite{shooting}.

Our proof of Theorem \ref{thm-main} uses crucially the fact that the class of equations \eqref{criticaldiracgen} is invariant under inversion in the unit sphere (\emph{Kelvin transform}). This reduces the problem of proving the existence of asymptotics to the problem of continuity at the origin. The latter will be solved by a bootstrap argument using some ideas of \cite{jannellisolimini}.

In Section \ref{sec:reghigher}, for a special class of non-linearities $h$ we will be able to prove more regularity and more precise asymptotics at infinity. In particular, for solutions of \eqref{criticaldirac} we will prove that $\psi\in C^\infty$ if $n=2$ and $\psi\in C^{1,\alpha}$ for any $\alpha<2/(n-1)$ if $n\geq 3$. The $C^\infty$ regularity for $n=2$ was previously shown by a different argument in \cite{remarkdirac}. A weaker version of regularity for $n\geq 3$, namely $C^{1,\alpha}$ with some unspecified $\alpha$, appears in \cite{isobecritical}. Moreover, we will use this regularity together with the Kelvin transform to establish higher order asymptotics at infinity. In particular, for \eqref{criticaldirac} in $n=2$ we will prove a complete asymptotic expansion at infinity. We refer the reader to that section for more details.

\medskip

We now describe a class of well-known solutions of \eqref{criticaldirac}. We fix a vector $\boldsymbol{n}\in \C^{N/2}$ with $\vert\boldsymbol{n}\vert=1$ and a parameter $\lambda>0$ and consider
\be\label{eq:standardbubble}
\psi(x)=\lambda^{-(n-1)/2} \begin{pmatrix} V(r/\lambda)\boldsymbol{n} \\ \I U(r/\lambda)\left(\frac{ x}{r}\cdot\boldsymbol{\sigma} \right)\boldsymbol{n} \end{pmatrix}
\ee
with
\begin{equation}
\label{eq:standardbubblescalar}
U(r) = n^{(n-1)/2} (1+r^2)^{-n/2} r \,,
\qquad
V(r) = n^{(n-1)/2} (1+r^2)^{-n/2} \,.
\end{equation}
These functions appear, for instance, in \cite{ammannmass}. A straightforward computation shows that they are solutions of \eqref{criticaldirac}. Moreover, they satisfy
\begin{equation}
\label{eq:standardbubblemodulus}
|\psi(x)| = \lambda^{-(n-1)/2} \sqrt{V(r/\lambda)^2+U(r/\lambda)^2} = \lambda^{-(n-1)/2} n^{(n-1)/2}(1+(r/\lambda)^2)^{-(n-1)/2} \,,
\end{equation}
which proves that the case $\Psi\neq 0$ in Theorem \ref{thm-main} can, in fact, occur.

The solutions \eqref{eq:standardbubble}, \eqref{eq:standardbubblescalar} are `ground state solutions' or `least energy solutions' of \eqref{criticaldirac} in the sense that any solution $\psi\not\equiv 0$ of \eqref{criticaldirac} satisfies
$$
\frac{1}{2}\int_{\R^{n}}\langle\cD\psi,\psi\rangle \,dx-\frac{1}{2^\sharp}\int_{\R^n}\vert\psi\vert^{2^\sharp}\,dx \geq \frac{1}{2n}\left( \frac{n}{2}\right)^{n} |\Sph^n|
$$
with equality exactly for \eqref{eq:standardbubble}, \eqref{eq:standardbubblescalar}. This bound was shown in \cite[Proposition 4.1]{isobecritical} and is based on inequalities by Hijazi \cite{hijazi} and B\"ar \cite{bar0} after mapping equation \eqref{criticaldirac} conformally to the sphere. (The equality statement made above is not explicitly stated in \cite{isobecritical}, but follows from the same arguments, taking the corresponding equality statements in Hijazi's and B\"ar's inequalities into account.) We also note that a simple extension of \cite{massless} to higher dimensions shows the existence of a non-trivial least energy solution for \eqref{criticaldirac}. This argument does not give the above minimal value, but has the advantage of working for a more general class of nonlinearities $h$.

Here we present a different characterization of the solutions \eqref{eq:standardbubble}, \eqref{eq:standardbubblescalar}. Given $\boldsymbol{n}\in \C^{N/2}$ with $\vert\boldsymbol{n}\vert=1$ we consider solutions of \eqref{criticaldirac} of the form
\be\label{soler}
\psi(x)=
\begin{pmatrix}
v(r)\boldsymbol{n} \\ 
\I u(r)\left(\frac{ x}{r}\cdot\boldsymbol{\sigma} \right)\boldsymbol{n} \end{pmatrix},
\qquad r=\vert x\vert,\qquad u,v:(0,\infty)\longrightarrow\R \,.
\ee
This form of solutions is sometimes \cite{nonrelativisticboussaid} called the \emph{Soler--Wakano-type} ansatz. It leads to the following ODE system
\begin{equation}\label{eq:radialsystem}
\left\{\begin{aligned}
u' + \frac{n-1}{r} u = v(u^2+v^2)^{1/(n-1)} \,,\\
v' = -u(u^2+v^2)^{1/(n-1)}  \,,
\end{aligned}\right.
\end{equation}
which needs to be supplemented by boundary conditions at the origin. To have $\psi$ regular at the origin it is natural to require $u(0)=0$ and then, to have a non-trivial solution, $v(0)\neq 0$ (see, e.g. \cite{massless}). We will see, however, that weaker boundary conditions suffice.

\begin{thm}\label{classification}
Let $n>1$ and let $u,v$ be real functions on $(0,\infty)$ satisfying \eqref{eq:radialsystem} as well as
$$
\lim_{r\to 0} r^{(n-1)/2} u(r) = \lim_{r\to 0} r^{(n-1)/2} v(r) = 0 \,.
$$
Then either $u=v=0$ or
$$
u(r) = \sigma \lambda^{-(n-1)/2} U(r/\lambda) \,,
\qquad
v(r) = \sigma \lambda^{-(n-1)/2} V(r/\lambda)
$$
for some $\lambda>0$ and $\sigma\in\{+1,-1\}$ with $U$ and $V$ from \eqref{eq:standardbubblescalar}.
\end{thm}

\begin{remark}\label{singsol}
The boundary conditions at the origin are necessary for the result to hold, since $u(r)=v(r)=\sqrt{(1/2)((n-1)/2)^{n-1}}\, r^{-(n-1)/2}$ is also a solution of the equation. Moreover, the same result holds, with the same proof, if the boundary condition at the origin is replaced by the condition
$$
\lim_{r\to\infty} r^{(n-1)/2} u(r) = \lim_{r\to\infty} r^{(n-1)/2} v(r) = 0 \,.
$$
at infinity. 
\end{remark}

We next discuss `excited state solutions' of \eqref{criticaldirac} and, more generally, of \eqref{criticaldiracgen} in the case $n=2$. When $n=2$, we have $N=2$ and we write $\psi=(\psi_1,\psi_2)$. We consider the nonlinearity in \eqref{criticaldiracgen} of the form
$$
h(z) = \begin{pmatrix}
\beta_1 |z_1|^2 + 2\beta_2|z_2|^2 & 0 \\
0 & \beta_1|z_2|^2 + 2\beta_2 |z_1|^2
\end{pmatrix}
$$
with given parameters $\beta_1,\beta_2>0$. As we will explain below, this particular nonlinearity arises in a problem from mathematical physics. Equation \eqref{criticaldiracgen} becomes the system
\begin{equation}\label{diracsystem}
\left\{\begin{aligned}
    (-\I\partial_{x_{1}} - \partial_{x_{2}})\psi_{2} &=(\beta_{1}\vert\psi_{1}\vert^{2}+2\beta_{2}\vert\psi_{2}\vert^{2})\psi_{1}  \,,\\
         (-\I \partial_{x_{1}} + \partial_{x_{2}})\psi_{1} &= (\beta_{1}\vert\psi_{2}\vert^{2}+2\beta_{2}\vert\psi_{1}\vert^{2})\psi_{2}\,.
\end{aligned}\right.
\end{equation}
Given $S\in\Z$ we look for solutions of \eqref{diracsystem} of the form
\begin{equation}\label{ansatzexcitedstates}
\psi(x) = \begin{pmatrix}
\psi_1(x) \\
\psi_2(x) \end{pmatrix}
=
\begin{pmatrix}
v(r) e^{\I S\theta} \\
\I u(r) e^{\I (S+1)\theta}
\end{pmatrix},\quad x=(r\cos\theta,r\sin\theta),\quad u,v:(0,\infty)\longrightarrow\R \,.
\end{equation}
Plugging this ansatz into \eqref{diracsystem} gives the system
\begin{equation}\label{system}
\left\{\begin{aligned}
    u'+\frac{S+1}{r}u &=v(\beta_1 v^2 + 2\beta_{2}u^{2})\,, \\ 
   v'-\frac{S}{r}v&=-u(\beta_{1}u^{2}+2\beta_{2}v^{2}) \,.
\end{aligned}\right.
\end{equation}
The following theorem shows that there is a unique (up to symmetries) solution of this system and provides precise asymptotics of $u$ and $v$. In particular, we see that solutions with $S\neq 0$ have a polynomially faster decay than the ground state solution \eqref{eq:standardbubble}, \eqref{eq:standardbubblescalar}. In particular, these are examples where the case $\Psi=0$ in Theorem \ref{thm-main} occurs.

\begin{thm}\label{2Dcase}
Let $\beta_1,\beta_2>0$ and $S\in\Z$ and put
$$
a= \left( \frac{|2S+1|}{\beta_1+2\beta_2} \right)^{1/2}
\qquad\text{and}\qquad
\tau = \sgn(S+1/2) \,.
$$
\begin{enumerate}
\item Let $(u,v)=(Q,P)$ be the solution of \eqref{system} with
$$
Q(1) = a
\qquad\text{and}\qquad
P(1) = \tau a \,.
$$
Then $(Q,P)$ exists globally and satisfies for some $\ell\in(0,\infty)$, if $S+1/2>0$,
$$
\lim_{r\to\infty} r^{S+1} Q(r) = \lim_{r\to0} r^{-S} P(r)=\ell
$$
and
$$
\lim_{r\to\infty} r^{3S+2} P(r) = \lim_{r\to 0} r^{-3S-1} Q(r) = \frac{\beta_1 \ell^3}{2(2S+1)} \,,
$$
and if $S+1/2<0$,
$$
\lim_{r\to 0} r^{S+1} Q(r) = - \lim_{r\to\infty} r^{-S} P(r) = \ell
$$
and
$$
\lim_{r\to 0} r^{3S+2} P(r) = - \lim_{r\to\infty} r^{-3S-1} Q(r) = - \frac{\beta_1 \ell^3}{2(2S+1)} \,,
$$
Moreover,
$$
\tau\, Q(r)P(r) >0
\qquad\text{for all}\ r>0
$$
and
$$
P(r)= (\tau/r) \, Q(1/r)
\qquad\text{for all}\ r>0 \,.
$$
\item If $(u,v)$ is a solution of \eqref{system} satisfying
$$
\lim_{r\to 0} r^{1/2} u(r) = \lim_{r\to 0} r^{1/2} v(r) = 0 \,,
$$
then there are $\lambda>0$ and $\sigma\in\{-1,+1\}$ such that
$$
u(r) = \sigma \lambda^{-1/2} Q(r/\lambda)
\qquad
v(r) = \sigma \lambda^{-1/2} P(r/\lambda)
\qquad\text{for all}\ r>0 \,.
$$
\end{enumerate}
\end{thm}

In the special case $\beta_1=2\beta_2=1$, that is, for \eqref{criticaldirac}, we will be able to write down all the solutions explicitly, see Theorem \ref{fasterexplicit}.

The methods that we develop to prove Theorems \ref{classification}, \ref{2Dcase} and \ref{fasterexplicit} have further applications, of which we mention two. First, the methods show uniqueness up to symmetries of solutions of the form \eqref{soler} for the more general class of equations \eqref{criticaldiracgen} also in dimensions $n\geq 3$ under suitable assumptions on the nonlinearity $h$. Second, the methods allow one to classify all functions of the form \eqref{soler} which satisfy \eqref{criticaldirac} in $\R^n\setminus\{0\}$. We already mentioned one such solution in Remark \ref{singsol}. This is relevant to the spinorial analogue of the singular Yamabe problem \cite{schoen2,korevaar}. Solutions are probably of the form $r^{-(n-1)/2}$ times a periodic function of $\ln r$. This seems to be a universal feature of conformally invariant equations which, for instance, has been recently verified for a fourth order equation \cite{frankkonig}.


\subsection{First motivation: Spinorial Yamabe and Br\'ezis--Nirenberg problems}

Equations of the form \eqref{criticaldirac} appear, for instance, in the blow-up analysis of solutions of the equation
\be\label{criticaldiracmanifold}
\cD\psi=\mu\psi+\vert\psi\vert^{2^\sharp -2}\psi \qquad\mbox{on}\ M \,,
\ee
where $(M,g,\Sigma)$ is a compact spin manifold, that is, a compact Riemannian manifold $(M,g)$ carrying a spin structure $\Sigma$ \cite{jost,diracoperators}. In that case one can define a Dirac operator $\cD$ and show that its $L^2$-spectrum is discrete and composed of eigenvalues of finite multiplicities accumulating at $\pm\infty$ (see, e.g., \cite{diracoperators,jost}).

In \eqref{criticaldiracmanifold}, $\mu\in\R$ is a parameter. For $\mu=0$ the equation is referred to as the \emph{spinorial Yamabe equation} and its study has been initiated by Ammann and collaborators \cite{ammann,ammannmass,spinorialanalogue,ammannsmallest}; see also \cite{nadineboundedgeometry, nadineconformalinvariant,raulot,maalaoui} and references therein. Equation \eqref{criticaldiracmanifold} with general $\mu\in\R$ is reminiscent of the Br\'ezis--Nirenberg problem \cite{brezisnirenberg} and has been studied, for instance, in \cite{isobecritical} and \cite{bartschspinorial}.

In particular, in \cite{isobecritical} Isobe proved the spinorial analogue of Struwe's theorem \cite{struwedecomposition} for the Br\'ezis--Nirenberg problem. To describe this in more detail, we note that solutions of \eqref{criticaldiracmanifold} are critical points of the functional
\be\label{functionalmanifold}
\cL(\psi)=\frac{1}{2}\int_{M}\langle\cD\psi,\psi\rangle d\operatorname{vol_g}-\frac{\mu}{2}\int_{M}\vert\psi\vert^2d\operatorname{vol_g}-\frac{1}{2^\sharp}\int_{\R^n}\vert\psi\vert^{2^\sharp}d\operatorname{vol_g}
\ee
defined on $H^{\frac{1}{2}}(\Sigma M)$, the space of $H^{\frac{1}{2}}$-sections of the spinor bundle $\Sigma M$ of the manifold. Here $d\operatorname{vol_g}$ stands for the volume measure of $(M,g)$. Then \cite[Theorem 5.2]{isobecritical} states that any Palais--Smale sequence $(\psi_n)_{n\in\mathbb{N}}\subseteq H^{\frac{1}{2}}(\Sigma M)$ for the functional $\cL$ satisfies
\be\label{psdecomposition}
\psi_{n}=\psi_\infty+\sum^{N}_{j=1}\omega^{j}_{n}+o(1) \qquad\mbox{in $H^{\frac{1}{2}}(\Sigma M)$},
\ee
where $\psi_{\infty}$ is the weak limit of $(\psi_{n})_n$ and the $\omega^{j}_{n}$ are suitably rescaled spinors obtained by mapping solutions to \eqref{criticaldirac} to spinors on the manifold $M$. In that sense, equation \eqref{criticaldirac} that we study in this paper describes bubbles in the spinorial Yamabe and Br\'ezis--Nirenberg problems.


\subsection{Second motivation: Effective equation for graphene}

Critical Dirac equations also appear as effective models for two-dimensional physical systems related to graphene. More precisely, if $V\in C^{\infty}(\R^{2},\R)$ possesses the symmetries of a honeycomb lattice, then, as proved in \cite{FWhoneycomb}, the dispersion bands of Schr\"{o}dinger operators of the form
$$
H=-\Delta+V(x)
\qquad\text{in}\ L^2(\R^2)
$$ 
exhibit generically conical intersections (the so-called \textit{Dirac points}). This leads to the appearance of the Dirac operator as an effective operator describing, for instance, the dynamics of wave packets spectrally concentrated around such conical degeneracies. Let $u_{0}(x)=u^{\varepsilon}_{0}(x)$ be a wave packet spectrally concentrated around a Dirac point, that is,
\be\label{concentrated}
u^{\varepsilon}_{0}(x)=\sqrt{\varepsilon}(\psi_{0,1}(\varepsilon x)\Phi_{1}(x)+\psi_{0,2}(\varepsilon x)\Phi_{2}(x))
\ee
where $\Phi_{j}$, $j=1,2$, are Bloch functions at a Dirac point and the functions $\psi_{0,j}$ are some (complex) amplitudes to be determined. One expects that the solution to the nonlinear Schr\"{o}dinger equation with parameter $\kappa\in\R\setminus\{0\}$,
\be\label{gp}
\I \partial_{t}u=-\Delta u+V(x)u+\kappa\vert u\vert^{2}u \,,
\ee
with initial conditions $u_0^\epsilon$ evolves to leading order in $\epsilon$ as a modulation of Bloch functions,
\be\label{approximate}
u^{\varepsilon}(t,x)\underset{\epsilon\rightarrow0^{+}}{\sim}\sqrt{\varepsilon}\left(\psi_{1}(\varepsilon t,\varepsilon x)\Phi_{1}(x)+\psi_{2}(\varepsilon t,\varepsilon x)\Phi_{2}(x) + \mathcal{O}(\varepsilon)\right).
\ee
As suggested by Fefferman and Weinstein in \cite{wavedirac} the modulation coefficients $\psi_{j}$ satisfy the following effective Dirac system,
\begin{equation}\label{effective}
\left\{\begin{aligned}
    \partial_{t}\psi_{1}+\overline{\lambda}(\partial_{x_{1}}+\I \partial_{x_{2}})\psi_{2} &=-\I \kappa(\beta_{1}\vert\psi_{1}\vert^{2}+2\beta_{2}\vert\psi_{2}\vert^{2})\psi_{1} \,, \\
     \partial_{t}\psi_{2}+\lambda(\partial_{x_{1}}-\I \partial_{x_{2}})\psi_{1} &=-\I \kappa(\beta_{1}\vert\psi_{2}\vert^{2} + 2\beta_{2}\vert\psi_{1}\vert^{2})\psi_{2} \,,
\end{aligned}\right.
\end{equation}
with
\be\label{beta}
\beta_{1}:=\int_{Y}\vert\Phi_{1}(x)\vert^{4}\,dx=\int_{Y}\vert\Phi_{2}(x)\vert^{4}\,dx \,,
\qquad
\beta_{2}:=\int_{Y}\vert\Phi_{1}(x)\vert^{2}\vert\Phi_{2}(x)\vert^{2}\,dx \,.
\ee
Here $Y$ denotes a fundamental cell of the lattice and $\lambda\in\C\setminus\{0\}$ is a coefficient related to the potential $V$. The large, but finite, time-scale validity of the Dirac approximation has been proved in \cite{FWwaves} in the linear case $\kappa=0$ for Schwartz class intial data (\ref{concentrated}). The case of cubic nonlinearities, corresponding to \eqref{gp} with $\kappa\neq 0$, is treated in \cite{arbunichsparber} for high enough Sobolev regularity $H^{s}(\R^{2})$ with $s>3$.

For stationary solutions (that is, $\partial_t\psi_1=\partial_t\psi_2=0$) we can write the system \eqref{effective} as
$$
\left( \alpha_1(-i\partial_{x_1}) + \alpha_2(-i\partial_{x_2}) \right) 
\begin{pmatrix}
\tilde\psi_2 \\ \psi_1
\end{pmatrix}
= \frac{|\kappa|}{|\lambda|} \begin{pmatrix}
( \beta_1 |\tilde\psi_2|^2 + 2\beta_2 |\psi_1|^2)\tilde\psi_2 \\
( \beta_1 |\psi_1|^2 + 2\beta_2 |\tilde\psi_2|^2)\psi_1
\end{pmatrix}
$$
with $\alpha_1$ and $\alpha_2$ from \eqref{eq:alpha2} and with
$$
\tilde\psi_2 = - \frac{\kappa}{|\kappa|}\, \frac{\lambda}{|\lambda|}\, \psi_2 \,.
$$
Thus, we arrive at \eqref{diracsystem} for the vector $(\tilde\psi_2,\psi_1)$ with coefficients $(|\kappa|/|\lambda|)\beta_j$ instead of $\beta_j$.


\subsection{Outline of the paper}

The proof of the Theorem \ref{thm-main} is achieved in several steps. First, we rewrite \eqref{criticaldiracgen} as an integral equation, as explained in Section \ref{sec:integral}. This allows us in Section \ref{sec:weakdecay} to prove boundedness and H\"older continuity of solutions. Then the desired asymptotics are proved in Section \ref{sec:kelvin} with the help of the Kelvin transform. Higher regularity and more precise asymptotics are the content of Section \ref{sec:reghigher}. Section \ref{sec-optimizers} is devoted to the proof of the classification in Theorem \ref{classification}. The remaining two sections deal with the faster decay for excited states in the two-dimensional case, first establishing an explicit family of solutions for $\beta_1=2\beta_2=1$ and then proving existence, uniqueness and asymptotics for general $\beta_1,\beta_2>0$.


\section{An integral equation}\label{sec:integral}

We begin with a Liouville-type lemma.

\begin{lemma}\label{liouvillespinors}
Let $p\geq 1$ and assume that $\psi\in L^p(\R^n,\C^N)$ satisfies $\mathcal D\psi=0$ in $\R^n$ in the sense of distributions. Then $\psi\equiv 0$.
\end{lemma}

\begin{proof}
The corresponding result for scalar functions $u$ satisfying $\Delta u=0$ is well-known; see, e.g., \cite{GrWu}. The present result can be obtained by a simple modification of those arguments. An even simpler way is to deduce it directly from the scalar result. Namely, if $\phi\in C_c^\infty(\R^n,\C^N)$, then, since $\mathcal D\psi=0$ in the sense of distributions and since $\mathcal D\phi\in C_c^\infty(\R^n,\C^N)$,
$$
-\int_{\R^N} \langle \Delta\phi, \psi\rangle \,dx = \int_{\R^N} \langle \mathcal D^2\phi, \psi\rangle \,dx 
= \int_{\R^N} \langle \mathcal D(\mathcal D\phi),\psi\rangle \,dx= 0 \,.
$$
This means that $\Delta\psi=0$ in the sense of distributions. Applying the scalar result to each component of $\psi$, we obtain the assertion.
\end{proof}

We now use the above lemma to rewrite \eqref{criticaldirac} as an integral equation. The \emph{Green's function} $\Gamma$ of the Dirac operator $\cD$ is given by
\be\label{greendirac}
\Gamma(x-y)=\frac{\I}{|\Sph^{n-1}|}\boldsymbol{\alpha}\cdot\frac{x-y}{\vert x-y\vert^{n}} \,,
\ee
where $\boldsymbol{\alpha}=(\alpha_{j})^{n}_{j=1}$. One easily checks that this function satisfies for each fixed $y\in\R^n$ the equation
\be\label{greenequation}
\cD_x\Gamma(x-y)=\delta(x-y)I_{N}
\qquad\text{in}\ \R_x^n
\ee
in the sense of distributions.

\begin{lemma}\label{integraleq}
If $\psi\in L^{2^\#}(\R^n,\C^N)$ solves \eqref{criticaldiracgen} in the sense of distributions, then
$$
\psi=\Gamma\ast(h(\cdot,\psi)\psi) \,.
$$
\end{lemma}

\begin{proof}
We note that $\Gamma\in L^{\frac{n}{n-1},\infty}$. The assumptions $\psi\in L^{2^{\sharp}}$ and \eqref{eq:hass} imply that $h(\cdot,\psi)\psi\in L^{\frac{2n}{n+1}}$ and therefore, by the weak Young inequality,
$$
\tilde\psi:=\Gamma\ast(h(\cdot,\psi)\psi)
$$
satisfies
$$
\tilde\psi\in L^{2^{\sharp}}(\R^n,\C^N) \,.
$$
Moreover, it is easy to see that
$$
\mathcal D\tilde\psi = h(x,\psi)\psi
\qquad\text{in}\ \R^n
$$
in the sense of distributions. This implies that
$$
\mathcal D(\psi-\tilde\psi)=0
\qquad\text{in}\ \R^n
$$
in the sense of distributions and therefore, by Lemma \ref{liouvillespinors}, $\psi-\tilde\psi=0$, as claimed.
\end{proof}


\section{Boundedness and regularity}\label{sec:weakdecay}

Throughout this section we consider a distributional solution $\psi\in L^{2^\#}(\R^n,\C^N)$ of \eqref{criticaldiracgen}. Our goal is to show that $\psi\in L^\infty(\R^n)\cap C^{0,\alpha}(\R^n)$ for any $\alpha<1$.

\begin{proposition}\label{bounded}
$\psi\in L^\infty(\R^n,\C^N)$
\end{proposition}

\begin{proof}
We denote $q=2^\#$.

\emph{Step 1.} We show that $\psi\in L^r$ for any $q\leq r<\infty$ with $r^{-1}\geq q^{-1}-n^{-1}$. (More explicitly, $\psi\in L^r$ for any $q\leq r<\infty$ if $n=2,3$ and $\psi\in L^r$ for any $q\leq r\leq 2n/(n-3)$ if $n\geq 4$.)

To prove this assertion we show that there is a constant $C>0$ such that for any $M>0$ we have
\begin{equation}
\label{eq:goalstrong}
S_M := \sup\left\{ \left|\int_{\R^n} \langle\phi,\psi\rangle\,dx \right|:\ \|\phi\|_{r'}\leq 1 \,,\ \|\phi\|_{q'}\leq M \right\} \leq C \,.
\end{equation}
This implies that
$$
\sup\left\{ \left|\int_{\R^n} \langle\phi,\psi\rangle\,dx \right|:\ \|\phi\|_{r'}\leq 1 \,,\ \phi\in L^{q'} \right\} \leq C \,,
$$
and therefore, by density and duality, $\psi\in L^{r}$, as claimed.

We now fix $M>0$ and aim at proving \eqref{eq:goalstrong}. Moreover, let $\eps>0$ be a parameter to be specified later. We claim that there is a function $f_\epsilon$ on $\R^n$, taking values in the $N\times N$ matrices, which is bounded and supported on a set of finite measure and satisfies
$$
\| h(\cdot,\psi(\cdot)) - f_\epsilon \|_n \leq \epsilon \,.
$$
Indeed, we set $f_\epsilon(x) = h(x,\psi(x))\mathbf{1}_{\{\delta\leq |\psi(x)|\leq \mu\}}$ for some $0<\delta\leq\mu<\infty$. Then
$$
\| h(x,\psi(x)) - f_\epsilon(x) \| \leq C |\psi(x)|^{2^\#-2} \left( \mathbf{1}_{\{|\psi(x)|<\delta\}} + \mathbf{1}_{\{|\psi(x)|> \mu\}} \right)
\qquad\text{for all}\ x\in\R^n
$$
and therefore
$$
\| h(\cdot,\psi(\cdot)) - f_\epsilon \|_n^n \leq C^n \left( \int_{\{|\psi(x)|<\delta\}}  |\psi(x)|^{2^\#}\,dx + \int_{\{|\psi(x)|> \mu\}}  |\psi(x)|^{2^\#}\,dx \right).
$$
Since $\psi\in L^{2^\#}$, the right side is $\leq\epsilon^n$ if $\delta>0$ is sufficiently small and $\mu<\infty$ is sufficiently large. Moreover, by \eqref{eq:hass}, $\|f_\epsilon(x)\|\leq C \mu^{2^\#-2}$ for all $x\in\R^n$ and
$$
|\{ x\in\R^n:\ \|f_\epsilon(x)\| >0\}| \leq |\{ x\in\R^n:\ |\psi(x)|\geq\delta\}|<\infty \,,
$$
proving the claimed properties.

We set $g_{\eps}:= h(\cdot,\psi(\cdot))-f_{\eps}$. Let $\phi\in L^{r'}\cap L^{q'}$ with $\|\phi\|_{r'}\leq 1$ and $\|\phi\|_{q'}\leq M$. We claim that
\begin{equation}
\label{eq:duality}
\int_{\R^n} \langle\phi,\psi\rangle \,dx = \int_{\R^n}\langle\phi, (\Gamma*(f_{\eps}\psi))\rangle\,dx + \int_{\R^n} \langle\chi_\eps,\psi\rangle\,dx
\end{equation}
with
$$
\chi_\eps:= h^* \Gamma*(g_{\eps}^*(\Gamma*\phi)) \,.
$$
Indeed, by the integral equation from Lemma \ref{integraleq} we have
$$\int_{\R^n} \langle\phi,\psi\rangle \,dx = \int_{\R^n}\langle\phi, (\Gamma*(f_{\eps}\psi))\rangle\,dx + \int_{\R^n}\langle\phi,(\Gamma*(g_{\eps}\psi))\rangle\,dx \,.$$
Using Fubini's theorem and the fact that for all $x,y\in\R^n$ with $x\neq y$, $\Gamma(x-y)$ is an anti-Hermitian matrix one rewrites the second term on the right side as follows,
\begin{align*}
\int_{\R^n}\langle\phi,\Gamma\ast(g_{\eps} \psi)\rangle\, dx
& =\int_{\R^n}\langle\phi(x),\int_{\R^n}\Gamma(x-y)(g_{\eps}(y) \psi(y))\, dy\rangle \,dx\\
&=\int_{\R^n}dy\int_{\R^n}dx\, \langle\phi (x),\Gamma(x-y)(g_{\eps} (y)\psi(y))\rangle\\
&=-\int_{\R^n}dy\int_{\R^n}dx\, \langle \Gamma(x-y)\phi(x),g_{\eps} (y)\psi(y)\rangle \\
&=-\int_{\R^n}dy\, \langle\int_{\R^n}\Gamma(x-y)\phi (x)\,dx, g_{\eps} (y)\psi(y)\rangle\\
&=\int_{\R^n}\langle(\Gamma\ast\phi) (y), g_{\eps} (y)\psi(y)\rangle \,dy \,.
\end{align*}
Using $\psi=\Gamma\ast(h \psi)$ in the last integral and applying the same argument as above, we obtain \eqref{eq:duality}.

We now estimate the two integrals appearing on the right side of \eqref{eq:duality}. We define $s$ by $s^{-1}=r^{-1}+n^{-1}$ and note that by assumption $1< s\leq q$. Using the H\"older inequality and the weak Young inequality we can estimate the first term on the right side of \eqref{eq:duality} as follows,
\begin{align*}
\left| \int_{\R^n}\langle\phi, (\Gamma*(f_\eps\psi))\rangle\,dx \right| & \leq \|\phi\|_{r'} \|\Gamma*(f_\eps\psi)\|_{r} \lesssim \|\phi\|_{r'} \|\Gamma\|_{n/(n-1),\infty} \|f_\eps\psi\|_{s} \\
& \leq \|\phi\|_{r'} \|\Gamma\|_{n/(n-1),\infty} \|f_\eps\|_{qs/(q-s)} \|\psi\|_q \,.
\end{align*}
Thus,
\begin{equation}
\label{eq:firstterm}
\left| \int_{\R^n}\langle\phi,(\Gamma*(f_\eps\psi))\rangle\,dx \right| \leq C_\epsilon \,,
\end{equation}
where $C_\eps$ depends, besides on $\eps$, only on $n$, $r$ and $\psi$ but not on $M$. 

We now claim that
\begin{equation}
\label{eq:goal2strong}
\|\chi_\eps\|_{r'} \leq C' \|g_{\eps}\|_{n} \|\phi\|_{r'} \,,
\qquad
\|\chi_\eps\|_{q'} \leq C' \|g_{\eps}\|_{n} \|\phi\|_{q'}
\end{equation}
with a constant $C'$ depending only on $n$ and $\|\psi\|_q$. Once this is shown, we infer from the definition of $S_M$ that
$$
\left| \int_{\R^n} \langle\chi_\eps,\psi\rangle\,dx \right| \leq C' \|g_{\eps}\|_{n} S_M \,.
$$
We now choose $\epsilon=1/(2C')$ and recall that $\|g_{\eps}\|_{n}\leq\epsilon$. In view of \eqref{eq:duality} and \eqref{eq:firstterm} we obtain
$$
\left| \int_{\R^n} \langle\phi,\psi\rangle\,dx \right|\leq C'' + \frac12 S_M \,,
$$
where $C''$ is $C_\epsilon$ with the above choice of $\epsilon$. Taking the supremum over all $\phi$, we obtain
$$
S_M \leq C'' +\frac12 S_M \,.
$$
We note that $S_M<\infty$ (in fact, since $\psi\in L^q$, we have $S_M\leq M \|\psi\|_q$). Therefore, the above inequality yields $S_M \leq 2C''$, which proves \eqref{eq:goalstrong}.

Thus, it remains to show \eqref{eq:goal2strong}. We have
\begin{align*}
\|\chi_\eps\|_{r'} & \leq \| h \|_{n} \|\Gamma*(g_\eps^*(\Gamma*\phi))\|_{s'}
\lesssim \| |\psi|^{q-2} \|_{n} \|\Gamma\|_{n/(n-1),\infty} \|g_\eps^*(\Gamma*\phi)\|_{r'} \\
& \leq \| |\psi|^{q-2} \|_{n} \|\Gamma\|_{n/(n-1),\infty} \|g_\eps\|_{n} \|\Gamma*\phi \|_{s'} \lesssim \| |\psi|^{q-2} \|_{n} \|\Gamma\|_{n/(n-1),\infty}^2 \|g_\eps\|_{n} \|\phi \|_{r'} \,.
\end{align*}
This proves the first inequality in \eqref{eq:goal2strong}. Similarly, we have
\begin{align*}
\|\chi_\eps\|_{q'} & \leq \| h \|_{n} \|\Gamma*(g_{\eps}^*(\Gamma*\phi))\|_{q}
\lesssim \| |\psi|^{q-2} \|_{n} \|\Gamma\|_{n/(n-1),\infty} \|g_{\eps}^*(\Gamma*\phi)\|_{q'} \\
& \leq \| |\psi|^{q-2} \|_{n} \|\Gamma\|_{n/(n-1),\infty} \|g_{\eps}\|_{n} \|\Gamma*\phi \|_{q} \lesssim \| |\psi|^{q-2} \|_{n} \|\Gamma\|_{n/(n-1),\infty}^2 \|g_{\eps}\|_{n} \|\phi \|_{q'} \,.
\end{align*}
This proves the second inequality in \eqref{eq:goal2strong} and concludes the proof of Step 1.

\emph{Step 2.} We show that if $\psi\in L^r$ for some $q<r<n(n+1)/(n-1)$, then $\psi\in L^s$ for $1/s= (n+1)/((n-1)r)- 1/n$.

Indeed, the assumption $\psi\in L^r$ implies that $h(\cdot,\psi)\psi\in L^{(n-1)r/(n+1)}$ and therefore, by the weak Young inequality, $\Gamma*(h(\cdot,\psi)\psi)\in L^s$, where $s$ is defined as above. (The weak Young inequality is applicable since $(n+1)/((n-1)r)-1/n>0$ by the assumed upper bound on $r$.) By Lemma \ref{integraleq} we obtain $\psi\in L^s$, as claimed.

\emph{Step 3.} We show that if $\psi\in L^r$ for some $r>n(n+1)/(n-1)$, then $\psi\in L^\infty$.

Indeed, the assumption implies that $h(\cdot,\psi)\psi\in L^s$ for $s=(n-1)r/(n+1)>n$. On the other hand, since $\psi\in L^{q}$, $h(\cdot,\psi)\psi\in L^{2n/(n+1)}$ and $2n/(n+1)<n$. Thus, writing $\Gamma$ as the sum of a function in $L^{s'}$ and one in $L^{2n/(n-1)}$, we obtain the assertion by H\"older's inequality.

\emph{Step 4.} Let us complete the proof of the proposition.

First assume $n=2,3$. Then according to Step 1, $\psi\in L^r$ for any $r<\infty$ and therefore, by Step 3, $\psi\in L^\infty$.

Now let $n\geq 4$. Define $r_1^{-1}=(n-3)/(2n)$ and inductively $r_{j+1}^{-1} = (n+1)/((n-1)r_j) - 1/n$ for $j\geq 1$. It is elementary to check that $(r_j^{-1})$ is a strictly decreasing sequence which tends to $-\infty$. Thus there is a largest $j$, say $J$, such that $r_j<n(n+1)/(n-1)$. By Step 1, $\psi\in L^{r_1}$ and, by applying Step 2 repeatedly, $\psi\in L^{r_{J+1}}$. Note that $r_{J+1}\geq n(n+1)/(n-1)$. If this inequality is strict, we infer from Step 3 that $\psi\in L^\infty$. Finally, if $r_{J+1}=n(n+1)/(n-1)$ we apply Step 2 with $r^{-1}=r_{J+1}^{-1}+\epsilon$ where $0<\epsilon<(n-1)^2/(n(n+1)^2)$. Then
$$
\frac 1s=\frac{n+1}{(n-1)r}-\frac 1n = \frac{n+1}{(n-1)r_{J+1}} + \frac{(n+1)\epsilon}{n-1} -\frac 1n = \frac{(n+1)\epsilon}{n-1}<\frac{n-1}{n(n+1)}
$$
and therefore $\psi\in L^s$ with $s>n(n+1)/(n-1)$. By Step 3, this implies again $\psi\in L^\infty$.
\end{proof}

\begin{proposition}\label{holder}
For any $\alpha<1$, $\psi\in C^{0,\alpha}(\R^n,\C^N)$.
\end{proposition}

\begin{proof}
By Proposition \ref{bounded} and the assumption $\psi\in L^{2^\#}$, we have $h(\cdot,\psi)\psi \in L^p$ for any $2n/(n+1)\leq p\leq \infty$. Therefore, by standard mapping properties of Riesz potentials, $\psi=\Gamma*(h(\cdot,\psi) \psi)\in C^{0,\alpha}$ for any $\alpha<1$.
\end{proof}


\section{The Kelvin transform and proof of Theorem \ref{thm-main}}\label{sec:kelvin}

We recall the definition of the unitary matrices $\mathcal{U}(x)$, $x\in\R^n\setminus\{0\}$, from \eqref{eq:defu}. Given a $\C^N$-valued function $\psi$ on $\R^n$, we define the \emph{Kelvin transform} for spinors
\be\label{kelvin}
\psi_\cK(x):= |x|^{-(n-1)} \mathcal{U}(x) \psi(x/|x|^2)
\qquad\text{for}\ x\in\R^n\setminus\{0\} \,.
\ee

Its basic properties are contained in the following

\begin{lemma}\label{inversiondirac} There holds
\be\label{dkelvin}
\mathcal{D} \psi_\cK = |x|^{-2} (\mathcal D\psi)_\cK \qquad\mbox{on $\R^n\setminus\{0\}$}\,.
\ee
Moreover
\be\label{quadkelvin}
\int_{\R^n} \langle\psi_\cK(x), (\mathcal D\psi_\cK)(x)\rangle\,dx = \int_{\R^n} \langle\psi(x),(\mathcal D\psi)(x)\rangle\,dx\,
\ee
and
\be\label{normkelvin}
 \int_{\R^n} |\psi_\cK(x)|^{2^\#}\,dx = \int_{\R^n} |\psi(x)|^{2^\#} \,dx\,.
\ee
\end{lemma}

\begin{proof}
We prove \eqref{dkelvin} for $n\geq3$. The case $n=2$ follows along the same lines. We fix $x\in\R^n\setminus\{0\}$ and compute, using \eqref{kelvin}, for any $j=1,...,n$,
\be\label{derivative}
\begin{split}
\partial_j\psi_\cK(x)&=-\frac{(n-1)x_j}{\vert x\vert^{n+1}}\, \mathcal U(x) \psi(x/\vert x\vert^2) \\
& \qquad +\vert x\vert^{-(n-1)}\begin{pmatrix}0 & -\I\Sum^n_{k=1}\sigma_k\left(\frac{\delta_{j,k}}{\vert x\vert} - \frac{x_k x_j}{\vert x\vert^3}\right) \\ \I\Sum^n_{k=1}\sigma_k\left(\frac{\delta_{j,k}}{\vert x\vert} - \frac{x_k x_j}{\vert x\vert^3} \right) & 0 \end{pmatrix}\psi(x/\vert x\vert^2) \\
& \qquad +\vert x\vert^{-(n-1)} \,\mathcal U(x)\Sum^n_{r=1}\partial_r\psi(x/\vert x\vert^2)\left(\frac{\delta_{j,r}}{\vert x\vert^2} - \frac{2x_r x_j}{\vert x\vert^4}\right)\,, 
\end{split}
\ee
where $\delta_{a,b}$ is the Kronecker symbol. We multiply \eqref{derivative} from the left by $-\I\alpha_j$, sum over $j=1,...,n$ and denote the resulting terms by
\be\label{summingup}
\cD\psi_\cK(x)=\Phi_1(x)+\Phi_2+\Phi_3(x)\,.
\ee

The anticommutation properties \eqref{anticommutation} and the definition \eqref{sigmaanticommutation} will be repeatedly used in the computations.
\smallskip

In order to rewrite $\Phi_1$, observe that 
\begin{equation}
\label{eq:computation}
(-\I\boldsymbol{\alpha}\cdot x)\mathcal{U}(x)= \vert x\vert \begin{pmatrix} I_{N/2} & 0 \\ 0 & -I_{N/2}  \end{pmatrix},
\end{equation}
where $I_{N/2}$ is the $(N/2)$-dimensional identity matrix. Then we get
\be\label{first}
\Phi_1(x)=-(n-1)\vert x\vert^{-n} \begin{pmatrix} I_{N/2} & 0 \\ 0 & -I_{N/2}  \end{pmatrix} \psi(x/|x|^2)\,.
\ee

We now turn to $\Phi_2$. Using
$$
\Sum^n_{j,k=1}\sigma_j\sigma_k\left(\frac{\delta_{j,k}}{\vert x\vert} - \frac{x_k x_j}{\vert x\vert^3}\right)= \left( \frac{n}{\vert x\vert}-\Sum^n_{j=1} \frac{x^2_j}{\vert x\vert^3} \right) I_{N/2} =\frac{n-1}{\vert x\vert}I_{N/2}
$$
we get
\be\label{second}
\Phi_2(x)=(n-1)\vert x\vert^{-n} \begin{pmatrix} I_{N/2} & 0 \\ 0 & -I_{N/2}  \end{pmatrix} \psi(x/|x|^2)  =-\Phi_1(x)\,.
\ee

Finally,
\be\label{third}\begin{split}
\Phi_3(x)& =\vert x\vert^{-(n-1)}\Sum^n_{j=1}(-\I\alpha_j)\mathcal U(x)\Sum^n_{r=1}\partial_r\psi(x/\vert x\vert^2)\left(\frac{\delta_{j,r}}{\vert x\vert^2} - \frac{2x_r x_j}{\vert x\vert^4}\right) \\
&=\vert x\vert^{-(n+1)}\Sum^n_{j=1}(-\I\alpha_j)\mathcal U(x)\partial_j \psi(x/\vert x\vert^2)\\
&\quad -2\vert x\vert^{-(n+3)}\Sum^n_{j=1}(-\I\alpha_j)x_j\mathcal U(x)\Sum^n_{r=1}x_r\partial_r\psi(x/\vert x\vert^2) \,.
\end{split}
\ee
The anticommutation properties of the matrices $\sigma_j$ give
\be\label{commuting}
(-i\alpha_j)\mathcal U(x)=\mathcal U(x)(-i\alpha_j)+\frac{2x_j}{\vert x	\vert}\begin{pmatrix}I_{N/2} &0 \\ 0 & -I_{N/2} \end{pmatrix} \,,
\ee
and we get
\be\label{last1}\begin{split}
&\vert x\vert^{-(n+1)}\Sum^n_{j=1}(-i\alpha_j)\mathcal U(x)\partial_j \psi(x/\vert x\vert^2) \\ &=\vert x\vert^{-2}(\cD\psi)_\cK(x)+2\vert x\vert^{-(n+2)}\begin{pmatrix}I_{N/2} &0 \\ 0 & -I_{N/2} \end{pmatrix} \Sum^n_{r=1}x_r\partial_r\psi(x/\vert x\vert^2)\,.
\end{split}
\ee
Using \eqref{eq:computation} the last term in \eqref{third} becomes
\begin{align*}
& -2\vert x\vert^{-(n+3)}\Sum^n_{j=1}(-\I\alpha_j)x_j\mathcal U(x)\Sum^n_{r=1}x_r\partial_r\psi(x/\vert x\vert^2) \\
& = -2\vert x\vert^{-(n+2)} \begin{pmatrix} I_{N/2} & 0 \\ 0 & -I_{N/2}  \end{pmatrix} \Sum^n_{r=1}x_r\partial_r\psi(x/\vert x\vert^2) \,.
\end{align*}
Thus, we have shown that
$$
\Phi_3(x) = \vert x\vert^{-2}(\cD\psi)_\cK(x) \,,
$$
which proves \eqref{dkelvin}.
\smallskip

Let us now prove \eqref{quadkelvin} and \eqref{normkelvin}. Because of the unitarity of $\mathcal U$, we have
$$
|\psi_\cK(x)| = |x|^{-(n-1)} |\psi(x/|x|^2)| \,.
$$
Therefore,
\begin{equation}
\label{eq:inversionlq}
\int_{\R^n} |\psi_\cK(x)|^{2^\#}\,dx = \int_{\R^n} |\psi(x/|x|^2)|^{2n/(n-1)}  |x|^{-2n}\,dx = \int_{\R^n} |\psi(x)|^{2n/(n-1)} \,dx \,.
\end{equation}

By \eqref{dkelvin} one finds
$$
\langle\psi_\cK(x), (\mathcal D\psi_\cK)(x)\rangle = |x|^{-2} \langle\psi_\cK(x),(\mathcal D\psi)_\cK(x)\rangle = |x|^{-2n} \langle\psi(x/|x|^2),(\mathcal D\psi)(x/|x|^2)\rangle
$$
and then
$$
\int_{\R^n} \langle\psi_\cK(x), (\mathcal D\psi_\cK)(x)\rangle\,dx = \int_{\R^n} \langle\psi(x),(\mathcal D\psi)(x)\rangle\,dx \,,
$$
thus concluding the proof of the lemma.
\end{proof}

We are now in the position to give the

\begin{proof}[Proof of Theorem \ref{thm-main}]
We aim at showing that $\psi_\cK$ also satisfies an equation of type \eqref{criticaldiracgen} and then, noting that $\psi_\cK\in L^{2^\#}(\R^n,\C^N)$ by \eqref{eq:inversionlq}, we can apply the regularity result from Proposition \ref{holder}.

According to \eqref{criticaldiracgen} and Lemma \ref{inversiondirac} we have
\begin{align*}
\mathcal D\psi_\cK & = |x|^{-2}(\mathcal D\psi)_\cK = |x|^{-2} (h(x,\psi)\psi)_\cK = |x|^{-n-1} \mathcal{U}(x) h(x/|x|^2,\psi(x/|x|^2))\psi(x/|x|^2) \\
& = \tilde h(x,\psi_\cK(x))\psi_\cK(x)
\qquad\text{in}\ \R^n\setminus\{0\} \,,
\end{align*}
where
$$
\tilde h(x,z) := |x|^{-2} \mathcal{U}(x)h(x/|x|^2,|x|^{n-1}\mathcal{U}(x)^*z)\mathcal{U}(x)^*
\qquad\text{for}\ x\in\R^n\setminus\{0\} \,,\ z\in\C^N \,.
$$
Note that assumption \eqref{eq:hass} implies that
$$
\| \tilde h(x,z)\| = |x|^{-2} \| h(x/|x|^2,|x|^{n-1}\mathcal{U}^*(x)z)\| \leq |x|^{-2} C \left| |x|^{n-1}\mathcal{U}^*(x)z \right|^{2/(n-1)} = C |z|^{2/(n-1)} \,,
$$
that is, $\tilde h$ satisfies also assumption \eqref{eq:hass}.

Let us show that $\psi_\cK$ satisfies the equation not only on $\R^n\setminus\{0\}$ but on $\R^n$, that is,
\begin{equation}
\label{eq:eqpsii}
\int_{\R^n} \langle \mathcal D \phi,\psi_\cK\rangle\,dx = \int_{\R^n} \langle \phi, \tilde h(x,\psi_\cK) \psi_\cK\rangle\,dx 
\qquad\text{for all}\ \phi\in C_c^\infty(\R^n,\C^N) \,.
\end{equation}
To prove this, let $\eta\in C^\infty(\R^n,\R)$ with $\eta\equiv 0$ in a neighborhood of the origin and $\eta\equiv 1$ outside the unit ball, and set $\eta_\epsilon(x):=\eta(x/\epsilon)$. Then, since $\eta_\epsilon\phi\in C_c^\infty(\R^n\setminus\{0\},\C^N)$,
$$
\int_{\R^n} \langle \mathcal D (\eta_\epsilon\phi),\psi_\cK\rangle\,dx = \int_{\R^n} \eta_\epsilon \langle \phi, \tilde h(x,\psi_\cK) \psi_\cK\rangle\,dx
$$
Since $\langle \phi, \tilde h(x,\psi_\cK) \psi_\cK\rangle\in L^1(\R^n)$ and $|\eta_\epsilon|\leq \|\eta\|_\infty$ with $\eta_\epsilon\to 1$ pointwise almost everywhere, dominated convergence implies that the right side converges to the right side of \eqref{eq:eqpsii}. On the left side we use $\mathcal D (\eta_\epsilon\phi)= \eta_\epsilon\mathcal D\phi - \I(\boldsymbol{\alpha}\cdot\nabla\eta_\epsilon)\phi$. The contribution coming from first term here converges to the left side of \eqref{eq:eqpsii} by a similar argument as before. Finally,
$$
\left| \int_{\R^n} \langle(\boldsymbol{\alpha}\cdot\nabla\eta_\epsilon)\phi,\psi_\cK\rangle\,dx \right| \leq \|\psi_\cK\|_{2^\#} \|\phi\|_\infty \|\nabla\eta_\epsilon\|_\infty |\{ |x|<\epsilon\}|^\frac{n+1}{2n} \lesssim \epsilon^\frac{n-1}{2}
$$
which tends to zero as $\epsilon\to 0$. This proves \eqref{eq:eqpsii}.

We now conclude from Proposition \ref{holder} that $\psi_\cK\in C^{0,\alpha}(\R^n)$ for any $\alpha<1$. In particular,
$$
\left| \psi_\cK(y) - \Psi \right| \leq C_\alpha |y|^\alpha
\qquad\text{for all}\ |y|\leq 1
$$
with $\Psi=\psi_\cK(0)$. Since $\mathcal U(x/|x|^2)^*=\mathcal U(x)$, this is equivalent to the bound in the theorem.
\end{proof}


\section{Higher regularity}\label{sec:reghigher}

In this section we assume that $H$ is a function defined on $\{z\in\C^N:\ |z|=1\}$ taking values in the $N\times N$-matrices. We assume that $H$ is H\"older continuous of order $\min\{2/(n-1),1\}$. Moreover, we assume that $H$ satisfies
$$
H(z) = \mathcal{U}(\omega) H(\mathcal{U}(\omega)^*z)\mathcal{U}(\omega)^*
\qquad\text{for any}\ z\in\C^N \,,\ \omega\in\R^n \ \text{with}\ |z|=|\omega|=1 \,.
$$
We set
$$
h(x,z) = h(z) = |z|^{2^\#-2} H(z/|z|) \,.
$$
Note that $h$ does not depend on $x$.

In this section we are interested in solutions of \eqref{criticaldiracgen} for $h$ of the above form and we shall show that both the regularity and the asymptotics results can be improved.

As an example of an allowed non-linearity we define for parameters $\beta_1,\beta_2\geq 0$,
\begin{equation}
\label{eq:exh}
H(z)= \begin{pmatrix}
\left(\beta_1 |z_1|^2 + 2\beta_2 |z_2|^2\right)^{2^\#-2} & 0 \\
0 & \left(\beta_1 |z_2|^2 + 2\beta_2 |z_1|^2\right)^{2^\#-2}
\end{pmatrix},
\ \ z=(z_1,z_2)\in\C^{N/2}\times\C^{N/2}.
\end{equation}
Indeed, the H\"older continuity follows from the elementary inequality $|(|\zeta|^2 + m^2)^{\alpha/2} -(|\zeta'|^2+m^2)^{\alpha/2}|\leq |\zeta-\zeta'|^\alpha$ and the symmetry condition follows by a simple computation. For $n=2$ this leads exactly to the non-linearity appearing in \eqref{diracsystem}. For general $n\geq 2$ and $\beta_1=2\beta_2=1$ we obtain the model equation \eqref{criticaldirac}.

\begin{proposition}\label{holderhigher}
Let $h$ be as above and let $\psi\in L^{2^\#}(\R^n,\C^N)$ be a weak solution of \eqref{criticaldiracgen}. Then
$$
\psi\in C^{1,\alpha}(\R^n,\C^N)
\qquad\text{for any}\ \alpha<\min\{2/(n-1),1\} \,.
$$
\end{proposition}

\begin{proof}
We first claim that, if $n\geq 3$,
$$
\|h(z)-h(z')\|\lesssim |z-z'|^{2/(n-1)}
\qquad\text{for all}\ z,z'\in\C^N
$$
and, if $n=2$,
$$
\|h(z)-h(z)\|\lesssim (|z|+|z'|) |z-z'|
\qquad\text{for all}\ z,z'\in\C^2
$$
Indeed, let $n\geq 3$ and assume without loss of generality that $|z'|\leq |z|$. Then, since $2^*-2=2/(n-1)\leq 1$,
\begin{align*}
\|h(z)-h(z')\| & \leq ||z|^{2^*-2}-|z'|^{2^*-2}| \|H(z/|z|)\| + |z'|^{2^*-2}\|H(z/|z|)-H(z'/|z'|)\| \\
& \leq \|H\|_{C^{1,2/(n-1)}} \left( \left||z|^{2^*-2}-|z'|^{2^*-2}\right|  + |z'|^{2^*-2} \left| z/|z| - z'/|z'| \right|^{2/(n-1)} \right)\\
& \leq \|H\|_{C^{1,2/(n-1)}} \left( \left||z|-|z'|\right|^{2^*-2}  + |z'|^{2^*-2} \left(|z-z'|/\sqrt{|z||z'|} \right)^{2^*-2} \right) \\
& \leq 2 \|H\|_{C^{1,2/(n-1)}} |z-z'|^{2^*-2} \,.
\end{align*}
The argument for $n=2$ is similar and is omitted.

We recall from Proposition \ref{holder} that $\psi\in C^{0,\alpha}$ for any $\alpha<1$. This, together with the above bounds on $h$ implies that $h(\psi)\psi \in C^{0,\alpha}$ for all $\alpha<1$ if $n=2$ and for all $\alpha<2/(n-1)$ if $n\geq 3$. Therefore, by mapping properties of Riesz potentials, $\psi=\Gamma*(h(\psi)\psi)\in C^{1,\alpha}$ for all $\alpha<1$ if $n=2$ and for all $\alpha<2/(n-1)$ if $n\geq 3$, as claimed.
\end{proof}

We obtain the following consequence about asymptotics.

\begin{corollary}\label{asymphigher}
Under the assumption of Proposition \ref{holderhigher} there are $\Psi,\Phi_1,\ldots,\Phi_n\in\C^N$ such that for any $\alpha<\min\{2/(n-1),1\}$ there is a $C_\alpha<\infty$ such that
$$
\left| \psi(x) - |x|^{-n+1}\mathcal{U}(x)\Psi + |x|^{-n} \sum_{j=1}^n \frac{x_j}{|x|} \mathcal{U}(x)\Phi_j \right| \leq C_\alpha |x|^{-n-\alpha}
\qquad\text{for all}\ |x|\geq 1 \,.
$$
\end{corollary}

\begin{proof}
We argue as in the proof of Theorem \ref{thm-main} by applying the regularity result to the Kelvin transform $\psi_\cK$. Indeed, Proposition \ref{holderhigher} implies that
$$
\left| \psi_\cK(y) - \Psi - \sum_{j=1}^n \Phi_j y_j \right|\leq C_\alpha |y|^{1+\alpha}
\qquad\text{for all}\ |y|\leq 1
$$
with $\Psi=\psi_\cK(0)$ and $\Phi_j = \partial_j\psi_\cK(0)$, which is equivalent to the bound in the corollary.

Therefore we only need to prove that the function $\tilde h(x,z)$ appearing in the proof of Theorem \ref{thm-main} is of the form allowed in Proposition \ref{holderhigher}. Using the assumptions on $h$, we have
$$
\tilde h(x,z) = |x|^{-2} \mathcal{U}(x) h(|x|^{n-1}\mathcal{U}(x)^* z) \mathcal{U}(x)^* = |z|^{2^\#-2} \mathcal{U}(x) H(\mathcal{U}(x)^* z/|z|) \mathcal{U}(x)^* = h(z) \,,
$$
so, in fact, the non-linearity is invariant under the Kelvin transform.
\end{proof}

\begin{proposition}\label{smooth}
Assume that $n=2$ and that $H$ is given by \eqref{eq:exh}. Let $\psi\in L^4(\R^2,\C^2)$ be a weak solution of \eqref{criticaldiracgen}. Then
$$
\psi\in C^\infty(\R^2,\C^2) \,.
$$
\end{proposition}

\begin{proof}
We know from Proposition \ref{holderhigher} that $\psi\in C^{1,\alpha}(\R^2,\C^2)$ for any $\alpha<1$. Since $h$ is smooth, we infer that also $h(\psi)\psi\in C^{1,\alpha}(\R^2,\C^2)$ for any $\alpha<1$ and therefore, by mapping properties of Riesz potentials, $\psi=\Gamma*(h(\psi)\psi)\in C^{2,\alpha}$ for all $\alpha<1$. Iterating, we obtain the proposition.
\end{proof}

Arguing in the same way as in the proof of Corollary \ref{asymphigher} we find that $\psi$ has a complete asymptotic expansion at infinity.

\begin{corollary}\label{asympcomplete}
Under the assumption of Proposition \ref{smooth} there are $\zeta_\alpha\in\C^2$, $\alpha\in\N_0^2$, such that for any $M\in\N$ there is a $C_M<\infty$ such that
$$
\left| \psi(x) - |x|^{-n+1} \mathcal{U}(x) \sum_{|\alpha|_1\leq M} \frac{x^\alpha}{|x|^{|\alpha|_1}} \zeta_\alpha \right| \leq C_M |x|^{-n-M}
\qquad\text{for all}\ |x|\geq 1 \,,
$$
where $x^\alpha=x_1^{\alpha_1}x_2^{\alpha_2}$ and $|\alpha|_1 =|\alpha_1|+|\alpha_2|$.
\end{corollary}

\begin{remark}
Under the assumptions of Proposition \ref{holderhigher}, but assuming in addition that $H$ is $C^\infty$, one has $\psi\in C^\infty(\R^n\setminus\{\psi=0\})$ for any $n\geq 2$. This follow by the same argument as in Proposition \ref{smooth} because $h(\psi)\psi$ has the same H\"older regularity as $\psi$ in compact subsets of $\{\psi\neq 0\}$. Translating this statement into asymptotics, we see that, if $\Psi\neq 0$ in Corollary \ref{asymphigher}, then $\psi$ has a complete asymptotic expansion at infinity similarly as in Corollary \ref{asympcomplete}.
\end{remark}

In connection with the previous remark we mention the deep results of \cite{bar} that the set $\{\psi=0\}$ has Hausdorff dimension at most $n-2$.


\section{Classification of `radial' solutions}
\label{sec-optimizers}

Our goal in this section is to prove Theorem \ref{classification} which classifies all solutions of \eqref{criticaldirac} of the form \eqref{soler}. We emphasize that once one has passed to the radial formulation \eqref{eq:radialsystem} the restriction that $n$ is an integer can be dropped. Our proof works for general real $n>1$.

\begin{proof}[Proof of Theorem \ref{classification}]
We pass to logarithmic variables and write
$$
u(r) = r^{-(n-1)/2} f(\ln r) \,,
\qquad
v(r) = r^{-(n-1)/2} g(\ln r)
$$
for functions $f,g$ defined on $\R$. The equations \eqref{eq:radialsystem} become
$$
f' + \frac{n-1}{2} f = g (f^2 + g^2)^{1/(n-1)} \,,
\qquad
g' - \frac{n-1}{2} g = - f (f^2+g^2)^{1/(n-1)}
$$
and the boundary conditions become
$$
\lim_{t\to-\infty} f(t) = \lim_{t\to-\infty} g(t) = 0 \,.
$$
We emphasize that the equations are now autonomous. Moreover, one easily checks that
$$
\mathcal E = -fg + \frac{1}{n} (f^2+g^2)^{n/(n-1)}
$$
is a constant. Since it tends to zero at $-\infty$, we conclude that
$$
fg = \frac{1}{n} (f^2+g^2)^{n/(n-1)}
\qquad\text{on}\ \R \,.
$$
We abbreviate $\rho = f^2+g^2$. Squaring the previous identity gives
$$
f^2 (\rho -f^2) = \frac{1}{n^2} \rho^{2n/(n-1)} \,.
$$
Solving for $f^2$ we obtain
$$
f^2 = \frac{1}{2} \left( \rho \pm \sqrt{\rho^2- \frac{4}{n^2}\rho^{2n/(n-1)}}\right),
\qquad
g^2 = \frac{1}{2} \left( \rho \mp \sqrt{\rho^2- \frac{4}{n^2}\rho^{2n/(n-1)}} \right).
$$
Note that this also implies that $\rho^2 \geq (4/n^2)\rho^{2n/(n-1)}$ on $\R$. The signs in the formulas for $f^2$ and $g^2$ are correlated. The signs may change but they may do so only at points where $\rho^2=(4/n^2)\rho^{2n/(n-1)}$.

Our next goal is to derive a differential equation for $\rho$. Using the differential equations for $f$ and $g$ we obtain
$$
(f^2+g^2)' = - (n-1)(f^2-g^2) \,,
$$
and inserting the above formulas for $f$ and $g$ we obtain
$$
\rho' = \mp (n-1)\sqrt{\rho^2- \frac{4}{n^2}\rho^{2n/(n-1)}} \,.
$$
This equation can be solved explicitly. On an interval where $\rho^2>(4/n^2)\rho^{2n/(n-1)}$ and where $f^2$ is given by the above formula with the $+$ sign, we find
$$
\rho(t) = \left( \frac{n}{2}\right)^{n-1} \cosh^{-n+1} (t-t_0)
$$
for some $t_0\in\R$; see Remark \ref{odesol} below for details. Choosing a maximal interval with these properties we deduce that this interval has necessarily the form $(t_0,\infty)$. Analogously one sees that any maximal interval where $\rho^2>(4/n^2)\rho^{n/(n-1)}$ and where $f^2$ is given by the above formula with the $-$ sign is of the form $(-\infty,t_1)$ for some $t_1\in\R$ and
$$
\rho(t) = \left( \frac{n}{2}\right)^{n-1} \cosh^{-n+1} (t-t_1)
$$
on this interval. We conclude that $t_0=t_1$ and therefore 
$$
\rho(t) = \left( \frac{n}{2}\right)^{n-1} \cosh^{-n+1} (t-t_0)
\qquad\text{for all}\ t\in\R \,,
$$
unless one has $\rho^2= (4/n^2)\rho^{2n/(n-1)}$ on all of $\R$, which means that either $\rho\equiv 0$ (and therefore $f\equiv g\equiv 0$) or $\rho = ((n-1)/2)^{n-1}$. In the latter case, from the formulas for $f^2$ and $g^2$ we see that $f^2=g^2=(1/2)((n-1)/2)^{n-1}$, but these functions do not satisfy the boundary conditions at $-\infty$, so this case is excluded. (Note that this corresponds to the singular solution in Remark \ref{singsol}.)

We return to the non-trivial case. Inserting the formula for $\rho$ into the above formulas for $f^2$ and $g^2$ we deduce that
$$
f^2 = \frac12 \left( \frac{n}{2}\right)^{n-1} e^{-(t-t_0)} \cosh^{-n} (t-t_0) \,,
\qquad
g^2 = \frac12 \left( \frac{n}{2}\right)^{n-1} e^{t-t_0} \cosh^{-n} (t-t_0) \,.
$$
(In fact, one checks these computations separately on the intervals $(-\infty,t_0)$ and $(t_0,\infty)$, where one knows the signs in the formulas for $f^2$ and $g^2$. The change in sign at $t_0$ is compensated by the fact that the expression
$$
\sqrt{\rho^2-(4/n^2)\rho^{2n/(n-1)}}=(n/2)^{n-1} \cosh^{-n}(t-t_0) |\sinh(t-t_0)|
$$
involves the absolute value of the $\sinh$.)

The formula $fg=(1/n)\rho^{n/(n-1)}$ together with the fact that $\rho$ never vanishes implies that $f$ and $g$ are either both positive or both negative. Thus, for some $\sigma\in\{+1,-1\}$,
$$
f=\sigma \sqrt{\frac12 \left( \frac{n}{2}\right)^{n-1}} e^{(t-t_0)/2} \cosh^{-n/2} (t-t_0) \,,
\quad
g = \sigma \sqrt{\frac12 \left( \frac{n}{2}\right)^{n-1}} e^{-(t-t_0)/2} \cosh^{-n/2} (t-t_0) \,.
$$
Changing back to the variable $r$ we have obtained the claimed formulas with $\lambda = e^{t_0}$.
\end{proof}

\begin{remark}\label{odesol}
In the proof above we solved the equation $\rho'=\mp(n-1)\sqrt{\rho^2-(4/n^2)\rho^{2n/(n-1)}}$. This can be done as follows, treating for instance the case with the $-$ minus. Let $\sigma(a)=\sqrt{1-(4/n^2)a^{2/(n-1)}}$ for $0< a<(2/n)^{-n+1}$. Then the equation reads $\rho'=-(n-1)\rho \sigma(\rho)$. We compute
$$
\frac{d\sigma(a)}{da} = - \frac1{n-1} (1-(4/n^2)a^{2/(n-1)})^{-1/2} (4/n^2) a^{2/(n-1)-1}= - \frac{1}{n-1} \frac{1-\sigma(a)^2}{a\sigma(a)} \,.
$$
Thus, using the equation for $\rho$,
$$
(\sigma(\rho))' = \frac{d\sigma}{da}|_{a=\rho} \rho' = 1-\sigma(\rho)^2 \,.
$$
Since $(1-b^2)^{-1}$ has anti-derivative $\text{artanh}\, b$, we obtain that for some $t_0$,
$$
t-t_0 = \mathrm{artanh}\, \sigma(\rho(t)) \,.
$$
That is,
$$
\tanh(t-t_0) = \sqrt{1-(4/n^2)\rho(t)^{2/(n-1)}} \,,
$$
which gives the claimed formula for $\rho(t)$.
\end{remark}


\section{Excited states in 2D}\label{sec:faster}

In this subsection we consider equation \eqref{criticaldirac} with $n=2$, which is the same as \eqref{diracsystem} with $\beta_{1}=1$ and $\beta_{2}=1/2$. We make the ansatz \eqref{ansatzexcitedstates} and arrive at the equations
\begin{equation}
\label{eq:excitedsystemexpl}
\begin{cases}
u' + \frac{S+1}{r} u & = v(u^2+v^2) \,, \\
v' - \frac{S}{r} v & = -u (u^2+v^2) \,,
\end{cases}
\end{equation}
which, of course, need to be supplemented with boundary conditions. We will prove the following classification result analogous to Theorem \ref{classification}.

\begin{thm}\label{fasterexplicit}
Let $S\in\Z$ and let $u,v$ be real functions on $(0,\infty)$ satisfying \eqref{eq:excitedsystemexpl} as well as
\be\label{eq:boundaryconditions}
\lim_{r\to 0} r^{1/2} u(r) = \lim_{r\to 0} r^{1/2} v(r) = 0 \,.
\ee
Then either $u=v=0$ or
$$
u(r) = \sigma \lambda^{-1/2} U(r/\lambda) \,,
\qquad
v(r) = \tau \lambda^{-1/2} V(r/\lambda)
$$
for some $\lambda>0$ and $\sigma,\tau\in\{+1,-1\}$, where
$$
U(r) = \sqrt{2|2S+1|} \frac{r^S}{r^{2S+1}+r^{-(2S+1)}} \,,
\qquad
V(r) = \sqrt{2|2S+1|} \frac{r^{-S-1}}{r^{2S+1}+r^{-(2S+1)}}
$$
and $\sigma=\tau$ if $2S+1>0$ and $\sigma=-\tau$ if $2S+1<0$.
\end{thm}

We emphasize that as in Theorem \ref{classification} we impose rather weak boundary conditions.

In our proof we do not use the fact that $S$ is an integer. Any real number $S$ works. For $S=-1/2$ the proof shows that $u=v=0$ is the only solution satisfying the boundary conditions.

\begin{proof}
We write again
$$
u(r) = r^{-1/2} f(\ln r) \,,
\qquad
v(r) = r^{-1/2} g(\ln r)
$$
for functions $u,v$ defined on $\R$. The equations become
$$
f' + \left( S+\frac{1}{2}\right) f = g (f^2 + g^2) \,,
\qquad
g' - \left(S +\frac{1}{2}\right) g = - f (f^2+g^2)
$$
and the boundary conditions become
$$
\lim_{t\to-\infty} f(t) = \lim_{t\to-\infty} g(t) = 0 \,.
$$
One easily checks that
$$
\mathcal E = -(2S+1) fg + \frac{1}{2} (f^2+g^2)^2
$$
is a constant. Since it tends to zero at $-\infty$, we conclude that
$$
fg = \frac{1}{2(2S+1)} (f^2+g^2)^2
\qquad\text{on}\ \R \,.
$$
We abbreviate $\rho = f^2+g^2$. Squaring the previous identity gives
$$
f^2 (\rho -f^2) = \frac{1}{4(2S+1)^2} \rho^4 \,.
$$
Solving for $f^2$ we obtain
$$
f^2 = \frac{1}{2} \left( \rho \pm \sqrt{\rho^2- \frac{1}{(2S+1)^2}\rho^4}\right),
\qquad
g^2 = \frac{1}{2} \left( \rho \mp \sqrt{\rho^2- \frac{1}{(2S+1)^2}\rho^4} \right).
$$
Note that this also implies that $\rho^2 \geq (1/(2S+1)^2)\rho^4$ on $\R$. The signs in the formulas for $f^2$ and $g^2$ are correlated. The signs may change but they may do so only at points where $\rho^2=(1/(2S+1)^2)\rho^4$.

Our next goal is to derive a differential equation for $\rho$. Using the differential equations for $f$ and $g$ we obtain
$$
(f^2+g^2)' = - (2S+1)(f^2-g^2) \,,
$$
and inserting the above formulas for $f$ and $g$ we obtain
$$
\rho' = \mp (2S+1) \sqrt{\rho^2- \frac{1}{(2S+1)^2}\rho^4} \,.
$$
This equation can be solved similarly as in Remark \ref{odesol} and we obtain, unless $\rho\equiv 0$,
$$
\rho(t) = |2S+1| \cosh^{-1}((2S+1)(t-t_0)) \,.
$$

Inserting this into the above formulas for $f^2$ and $g^2$ we deduce that
$$
f^2 = \frac{|2S+1|}2 e^{(2S+1)(t-t_0)} \cosh^{-2}((2S+1)(t-t_0))
$$
and
$$
g^2 = \frac{|2S+1|}2 e^{-(2S+1)(t-t_0)} \cosh^{-2}((2S+1)(t-t_0)) \,.
$$
(Here one has to distinguish according to whether $2S+1$ is positive or negative.) When $2S+1>0$ the formula $fg=(1/2(2S+1))\rho^2$ together with the fact that $\rho$ never vanishes implies that $f$ and $g$ are either both positive or both negative. Thus, for some $\sigma\in\{+1,-1\}$,
$$
f=\sigma \sqrt{\frac{|2S+1|}2} e^{(2S+1)(t-t_0)/2} \cosh^{-1} ((2S+1)(t-t_0))
$$
and
$$
g = \sigma \sqrt{\frac{|2S+1|}2} e^{-(2S+1)(t-t_0)/2} \cosh^{-1} ((2S+1)(t-t_0)) \,.
$$
Changing back to the variable $r$ we have obtained the claimed formulas with $\lambda = e^{t_0}$. In case $2S+1<0$ is similar, but $f$ and $g$ have opposite signs.
\end{proof}


\section{Faster decay for excited states}\label{sec:2Dgen}

The aim of this section is to prove Theorem \ref{2Dcase} concerning \eqref{diracsystem} with parameters $\beta_1,\beta_2>0$. The case $\beta_1=2\beta_2=1$ was treated in the previous section and the case $\beta_1=2\beta_2$ can be reduced to the former by scaling.

When $\beta_1\neq 2\beta_2$ we cannot provide explicit solutions, but we can still prove existence and uniqueness (up to symmetries) of a solution and can study its asymptotic behavior rather precisely.

\begin{proof}[Proof of Theorem \ref{2Dcase}]
\emph{Step 1. Introducing logarithmic variables.}
We set again
\be\label{uf}
u(r) = r^{-1/2} f(\ln r) \,,
\qquad
v(r) = r^{-1/2} g(\ln r)
\ee
for functions $f,g$ defined on $\R$, so that system \eqref{system} becomes
\begin{equation}\label{hs}
\begin{cases}
f' + \left( S+\frac{1}{2}\right) f & = g (2\beta_{2}f^2 + \beta_{1}g^2) \,, \\
g' - \left(S +\frac{1}{2}\right) g & = - f (\beta_{1}f^2+2\beta_{2}g^2)
\end{cases}
\end{equation}
and the boundary conditions in the theorem read
\be\label{-infty}
\lim_{t\to-\infty} f(t) = \lim_{t\to-\infty} g(t) = 0 \,.
\ee
One easily checks that
\be\label{hamiltonian}
\mathcal{E} = \frac{\beta_{1}}{4}(f^{4}+g^{4})+\beta_{2}f^{2}g^{2}-\left(S+\frac{1}{2}\right)fg
\ee
is constant. The boundary conditions \eqref{-infty} imply that $\mathcal E=0$, that is,
\be\label{identity}
\frac{\beta_{1}}{4}(f^{4}+g^{4})+\beta_{2}f^{2}g^{2}=\left(S+\frac{1}{2}\right)fg.
\ee
This implies, in particular, that $f(t)\neq 0$ for all $t\in\R$ and $g(t)\neq 0$ for all $t\in\R$, unless $f\equiv g\equiv 0$. (Indeed, if $f(t_0)=0$, then \eqref{identity} implies $g(t_0)=0$ and then \eqref{hs} implies $f\equiv g\equiv 0$. The argument for $g$ is similar.) Moreover, it implies that
$$
\tau\, f(t)g(t) > 0
\qquad\text{for all}\ t\in\R \,.
$$

\emph{Step 2. Monotonicity of the angle.} We shall show that, if $(f,g)\neq(0,0)$ is a solution of \eqref{hs} with $\mathcal E=0$, then $(f,g)$ is global and the limits
$$
\theta_\pm := \lim_{t\to\pm\infty} \arctan\frac{g(t)}{f(t)}
$$
exist and satisfy $\theta_+<\theta_-$.

Indeed, the fact that $\mathcal E=0$ on the maximal interval of existence easily implies that the solution is global. Moreover, as remarked in the previous step, $\mathcal E=0$ implies that $f$ and $g$ never vanish and therefore we can introduce
$$
\theta(t) = \arctan\frac{g(t)}{f(t)} \,.
$$
Using \eqref{hs} and \eqref{identity} we compute
\begin{align*}
\theta'= \frac{g'f-f'g}{f^2+g^2} = \frac{(2S+1)gf - \beta_1 (f^4+g^4) - 4\beta_2  f^2 g^2}{f^2+g^2}= \frac{- (\beta_1/2) (f^4+g^4) - 2\beta_2  f^2 g^2}{f^2+g^2} <0 \,.
\end{align*}
This proves the claim.

\emph{Step 3. Asymptotics of solutions.} We shall show that any solution $(f,g)$ of \eqref{hs} with $\mathcal E=0$ is global and satisfies \eqref{-infty} and
\be\label{+infty}
\lim_{t\to\infty} f(t) = \lim_{t\to\infty} g(t) = 0 \,.
\ee

The global existence was already shown in the previous step. We shall deduce the asymptotic behavior from the Poincar\'e--Bendixson theorem in the form given, for instance, in \cite[Theorem 7.16]{Teschl}. Let
$$
\omega_\pm = \{ (x,y)\in\R^2:\ \text{for some}\ t_n\to\pm\infty \,, (f(t_n),g(t_n))\to (x,y)\} \,.
$$
Since the set $\{ (x,y):\ (\beta_1/4)(x^4+y^4) +\beta_2 x^2y^2=(S+1/2)xy\}$ is compact, it is easy to see that the sets $\omega_\pm$ are non-empty, compact and connected \cite[Lemma 6.6]{Teschl}. According to Poincar\'e--Bendixson, for each one of the signs $\pm$, one of the following three alternatives holds: (a) $\omega_\pm$ is a fixed point, (b) $\omega_\pm$ is a regular periodic orbit, (c) $\omega_\pm$ consists of fixed points and non-closed orbits connecting these fixed points.

A simple computation shows that the only constant solution of \eqref{hs} with $\mathcal E=0$ is $(f,g)\equiv(0,0)$. Thus, if alternative (a) holds for both signs $\pm$, then we are done. Let us rule out (b) and (c). Note that in both cases (b) and (c), the limiting periodic orbit and the limiting homoclinic orbits, if they would exist, would have $\mathcal E=0$.

According to Step 2, there are no non-trivial periodic solutions of \eqref{hs} with $\mathcal E=0$ (because for a non-trivial periodic solution $(\tilde f,\tilde g)$, $\arctan (\tilde g/\tilde f)$ does not have a limit). This rules out (b).

According to Step 2, there are $\theta_\pm$ such that
$$
\omega_\pm \subset\{ (r\cos\theta_\pm,r\sin\theta_\pm):\ r\geq 0\} \,.
$$
Thus, in order to rule out (c), it suffices to rule out the existence of a non-trivial solution $(\tilde f,\tilde g)$ of \eqref{hs} with $(\tilde f(t),\tilde g(t))\to (0,0)$ for $|t|\to\infty$ and such that $\arctan (\tilde g(t)/\tilde f(t))=\theta_\pm$ for all $t$. But this is again ruled out by Step 2. This completes the proof of the assertion.

\emph{Step 4. Existence of a homoclinic orbit.} Let $a$ and $\tau$ be as in the theorem and consider the solution $(f,g)=(q,p)$ of \eqref{hs} with initial values
$$
q(0)=\tau\, p(0) = a \,.
$$
We shall show that this solution is global and satisfies the asymptotic conditions \eqref{-infty} and \eqref{+infty}.

Indeed, by definition of $a$, identity \eqref{identity} is satisfied at $t=0$. Therefore, the solution has $\mathcal E=0$. The rest now follows from Step 3.

\emph{Step 5. Exponential decay.} We shall show that for any solution $(f,g)$ of \eqref{hs} satisfying \eqref{-infty} there is a constant $C$ such that
$$
(f^2 + g^2)^{1/2} \leq C e^{-|S+1/2| |t|}
\qquad\text{for all}\ t\in\R \,.
$$

Indeed, we compute, using \eqref{hs},
$$
(f^2+g^2)' = 2(ff'+gg') = (-(2S+1)+(2\beta_2-\beta_1)fg)(f^2-g^2) \,,
$$
$$
(f^2-g^2)' = 2(ff'-gg') = (-(2S+1)+2(2\beta_2+\beta_1)fg)(f^2+g^2)
$$
and
$$
(fg)' = f'g+fg'= -\beta_1(f^4-g^4) \,.
$$
This implies that
\begin{align*}
(f^2+g^2)'' & = (-(2S+1)+(2\beta_2-\beta_1)fg)(-(2S+1)+2(2\beta_2+\beta_1)fg)(f^2+g^2) \\
& \quad - (2\beta_2-\beta_1)\beta_1(f^4-g^4)(f^2-g^2)\,.
\end{align*}
We set $\psi=f^2+g^2$ and write the previous equation as
$$
-\psi'' + V\psi = - (2S+1)^2\psi
$$
with
$$
V = -(2S+1)( 6\beta_2+\beta_1)fg + (2\beta_2-\beta_1)2(2\beta_2+\beta_1)f^2 g^2 - (2\beta_2-\beta_1)\beta_1 (f^2-g^2)^2 \,.
$$
By \eqref{-infty} we have $\mathcal E=0$ and therefore, by Step 3, $V(t)\to 0$ as $|t|\to\infty$. By a standard comparison argument this implies that for any $0<\epsilon\leq (2S+1)^2$ there is a $C_\epsilon$ such that 
\begin{equation}
\label{eq:decayepsilon}
\psi(t) \leq C_\epsilon e^{-\sqrt{(2S+1)^2-\epsilon} |t|}
\qquad
\text{for all}\ t\in\R \,.
\end{equation}
For the sake of completeness we provide the details of this argument. Given $0<\epsilon\leq (2S+1)^2$ we choose $T_\epsilon<\infty$ such that $V(t)\geq -\epsilon$ for $t\geq T_\epsilon$. The function
$$
\phi(t) =\psi(t) - \psi(T_\epsilon) e^{-\sqrt{(2S+1)^2-\epsilon}(t-T_\epsilon)}
$$
satisfies $\phi(T_\epsilon)=0$, $\lim_{t\to\infty} \phi(t)=0$ and
$$
\phi''\geq ((2S+1)^2-\epsilon)\phi
\qquad\text{in}\ (T_\epsilon,\infty) \,.
$$
By the maximum principle, this implies that $\phi\leq 0$ in $[T_\epsilon,\infty)$. Similarly, one proves a bound near $-\infty$ and the remaining bound is obtained by continuity. This proves \eqref{eq:decayepsilon}.

Because of the decay \eqref{eq:decayepsilon} we can apply the Green's function to the equation for $\psi$ and obtain
$$
\psi(t) = - \frac{1}{2|2S+1|} \int_\R e^{-|2S+1| |t-t'|} V(t')\psi(t')\,dt' \,.
$$
Using this equation and the apriori bound \eqref{eq:decayepsilon} it is easy to obtain the claimed bound for $\psi$.

\emph{Step 6. Asymptotic behavior of $f$ and $g$.} Again, we let $(f,g)\not\equiv (0,0)$ be a solution of \eqref{hs} satisfying \eqref{-infty}. We shall show that
$$
\ell := \lim_{t\to\tau\infty} e^{(S+1/2) t} f(t)
\ \text{and}\
\ell' := \lim_{t\to -\tau\infty} e^{-(S+1/2) t} g(t)
\quad\text{exist and are non-zero and finite}
$$
and that
$$
\lim_{t\to\tau\infty} e^{3(S+1/2)t} g(t) = \frac{\beta_1 \ell^3}{4(S+1/2)} 
\quad\text{and}\quad
\lim_{t\to-\tau\infty} e^{-3(S+1/2)t} f(t) = \frac{\beta_1 \ell'^3}{4(S+1/2)}
\,.
$$

Let us prove this in case $S+1/2>0$ (so $\tau=+1$), the case $S+1/2<0$ being similar. The function $F(t)= e^{(S+1/2)t} f(t)$ satisfies
$$
F'(t) = e^{(S+1/2)t} ( f' + (S+1/2)f) = e^{(S+1/2)t} g(2\beta_2 f^2+\beta_1 g^2) \,.
$$
As shown in Step 1, either $f$ and $g$ are both positive or both negative. Thus, either $F$ is positive and increasing or it is negative and decreasing. Since it is bounded by Step 5, it tends in any case to a finite, non-zero limit $\ell$. This proves the first assertion.

The function $G(t)=e^{-(S+1/2)t}g(t)$ satisfies
\begin{equation}
\label{eq:Geq}
G'(t) = e^{-(S+1/2)t}(g'-(S+1/2)g)= -e^{-(S+1/2)t} f (\beta_1 f^2+2\beta_2 g^2) \,.
\end{equation}
For the sake of simplicity we now assume that $f$ and $g$ are both positive. The case where both are negative is treated similarly. Given $0<\epsilon\leq\ell$ there is a $t_\epsilon<\infty$ such that $f(t)\geq (\ell-\epsilon) e^{-(S+1/2)t}$ for $t\geq t_\epsilon$. We bound the right side of \eqref{eq:Geq} and get
$$
G'(t) \leq -\beta_1 (\ell-\epsilon)^3 e^{-4(S+1/2)t}
\qquad\text{for all}\ t\geq t_\epsilon \,,
$$
and, since $G(t)\to 0$ as $t\to\infty$ by Step 5,
$$
G(t) = - \int_t^\infty G'(s)\,ds \geq \beta_1 (\ell-\epsilon)^3 \int_t^\infty e^{-4(S+1/2)s}\,ds = \frac{\beta_1(\ell-\epsilon)^3}{4(S+1/2)} e^{-4(S+1/2)t}
\ \ \text{for all}\ t\geq t_\epsilon
$$
and
$$
g(t) \geq \frac{\beta_1(\ell-\epsilon)^3}{4(S+1/2)} e^{-3(S+1/2)t}
\qquad\text{for all}\ t\geq t_\epsilon \,.
$$
This is the desired asymptotic lower bound on $g(t)$. The proof of the upper bound is similar, but slightly more complicated. Using the bounds from Step 5 in \eqref{eq:Geq}, we get
$$
G'(t) \geq - \const e^{-4(S+1/2)t}
\qquad\text{for all}\ t \in\R
$$
and therefore, by a similar argument as before,
$$
g(t) \leq \const e^{-3(S+1/2)t}
\qquad\text{for all}\ t\in\R \,.
$$
Now again given $\epsilon>0$ there is a $t_\epsilon'<\infty$ such that $f(t)\leq (\ell+\epsilon)e^{-(S+1/2)t}$ for $t\geq t_\epsilon'$. Inserting this and the previous bound on $g$ in the equation for $G'$ we obtain
$$
G'(t) \geq -\beta_1(\ell+\epsilon)^3 e^{-4(S+1/2)t} - \const e^{-8(S+1/2)t}
\qquad\text{for all}\ t\geq t_\epsilon'
$$
and therefore by integration similarly as before
$$
g(t) \leq \frac{\beta_1(\ell+\epsilon)^3}{4(S+1/2)} e^{-3(S+1/2)t} + \const e^{-7(S+1/2)t}
\qquad\text{for all}\ t\geq t_\epsilon' \,.
$$
This proves the claimed asymptotics for $g$ as $t\to\infty$.

In order to obtain the asymptotics of $f$ and $g$ for $t\to-\infty$, we note that the pair $(g(-t),f(-t))$ solves \eqref{hs}. (Note that we have reversed the roles of $f$ and $g$.) Therefore, applying the previous statement to this solution we obtain the claimed asymptotics for $t\to-\infty$.

\emph{Step 7. Uniqueness.} We show that the non-trivial solution of \eqref{hs} satisfying \eqref{-infty} is unique, up to translation and a sign change.

We give the argument only for $S+1/2>0$, the case $S+1/2<0$ being similar. We know from Step 6 that
$$
\lim_{t\to\infty} \frac{g(t)}{f(t)} = 0 = \lim_{t\to-\infty} \frac{f(t)}{g(t)} \,.
$$
Thus, by continuity there is a $t_0\in\R$ such that $b:=f(t_0)=g(t_0)$. Assumption \eqref{-infty} implies $\mathcal E=0$ and therefore
$$
\frac{\beta_1}{4} 2b^4 + \beta_2 b^4 - (S+1/2) b^2 = 0 \,.
$$
Thus, $b\in\{0,-a,+a\}$ with $a$ defined in the theorem. Since the solution is non-trivial, we have $b\neq 0$. If $(q,p)$ denotes the solution from Step 4, then by uniqueness of the solution of an initial value problem we have $(f(t),g(t))=(q(t-t_0),p(t-t_0))$ if $b=a$ and $(f(t),g(t))=(-q(t-t_0),-p(t-t_0))$ if $b=-a$. This proves the above uniqueness claim.

\emph{Step 8. Conclusion of the proof.} We now prove all the statements of Theorem \ref{2Dcase} translated to logarithmic variables.

Let $(p,q)$ be the solution from Step 4 which we already know is global and satisfies \eqref{-infty}. Therefore Step 6 describes the asymptotic behavior of this solution. The fact that $\tau q(t)p(t)>0$ for all $t\in\R$ was already noted in Step 1 and the fact that $p(t)=\tau\,q(-t)$ for all $t\in\R$ follows from the fact that $(\tau\,q(-t),\tau\, p(-t))$ is a solution of \eqref{hs} with the same values at $t=0$ as $(p(t),q(t))$. This concludes the proof of part (1) of the theorem. Part (2) follows immediately from Step 7.
\end{proof}



\bibliographystyle{siam}
\bibliography{DecayDirac}

\begin{thebibliography}{10}

\bibitem{ammann}
{\sc B.~Ammann}, {\em A variational problem in conformal spin geometry},
  Habilitationsschift, Universit\"{a}t Hamburg, 2003.

\bibitem{ammannsmallest}
\leavevmode\vrule height 2pt depth -1.6pt width 23pt, {\em The smallest {D}irac
  eigenvalue in a spin-conformal class and cmc immersions}, Comm. Anal. Geom.,
  17 (2009), pp.~429--479.

\bibitem{spinorialanalogue}
{\sc B.~Ammann, J.-F. Grosjean, E.~Humbert, and B.~Morel}, {\em A spinorial
  analogue of {A}ubin's inequality}, Math. Z., 260 (2008), pp.~127--151.

\bibitem{ammannmass}
{\sc B.~Ammann, E.~Humbert, and B.~Morel}, {\em Mass endomorphism and spinorial
  {Y}amabe type problems on conformally flat manifolds}, Comm. Anal. Geom., 14
  (2006), pp.~163--182.

\bibitem{arbunichsparber}
{\sc J.~Arbunich and C.~Sparber}, {\em Rigorous derivation of nonlinear {D}irac
  equations for wave propagation in honeycomb structures}, J. Math. Phys., 59
  (2018), pp.~011509, 18.

\bibitem{bar0}
{\sc C.~B\"{a}r}, {\em Lower eigenvalue estimates for {D}irac operators}, Math.
  Ann., 293 (1992), pp.~39--46.

\bibitem{bar}
\leavevmode\vrule height 2pt depth -1.6pt width 23pt, {\em Zero sets of
  solutions to semilinear elliptic systems of first order}, Invent. Math., 138
  (1999), pp.~183--202.

\bibitem{bartschspinorial}
{\sc T.~{Bartsch} and T.~{Xu}}, {\em {A spinorial analogue of the
  Brezis-Nirenberg theorem involving the critical Sobolev exponent}}, ArXiv
  e-prints,  (2018).

\bibitem{berthiergeorgescu}
{\sc A.~Berthier and V.~Georgescu}, {\em On the point spectrum of {D}irac
  operators}, J. Funct. Anal., 71 (1987), pp.~309--338.

\bibitem{shooting}
{\sc W.~Borrelli}, {\em Stationary solutions for the 2{D} critical {D}irac
  equation with {K}err nonlinearity}, J. Differential Equations, 263 (2017),
  pp.~7941--7964.

\bibitem{massless}
\leavevmode\vrule height 2pt depth -1.6pt width 23pt, {\em Weakly localized
  states for nonlinear {D}irac equations}, Calc. Var. Partial Differential
  Equations, 57 (2018), p.~57:155.

\bibitem{spectralnabilecomech}
{\sc N.~Boussa\"id and A.~Comech}, {\em On spectral stability of the nonlinear
  {D}irac equation}, J. Funct. Anal., 271 (2016), pp.~1462--1524.

\bibitem{nonrelativisticboussaid}
\leavevmode\vrule height 2pt depth -1.6pt width 23pt, {\em Nonrelativistic
  asymptotics of solitary waves in the {D}irac equation with {S}oler-type
  nonlinearity}, SIAM J. Math. Anal., 49 (2017), pp.~2527--2572.

\bibitem{brezisnirenberg}
{\sc H.~Br\'ezis and L.~Nirenberg}, {\em Positive solutions of nonlinear
  elliptic equations involving critical {S}obolev exponents}, Comm. Pure Appl.
  Math., 36 (1983), pp.~437--477.

\bibitem{cassano}
{\sc B.~Cassano}, {\em {Sharp exponential localization for eigenfunctions of
  the Dirac Operator}}, ArXiv e-prints,  (2018).

\bibitem{FWhoneycomb}
{\sc C.~L. Fefferman and M.~I. Weinstein}, {\em Honeycomb lattice potentials
  and dirac points}, J. Amer. Math. Soc., 25 (2012), pp.~1169--1220.

\bibitem{wavedirac}
\leavevmode\vrule height 2pt depth -1.6pt width 23pt, {\em Waves in honeycomb
  structures}, Journ\'{e}es \'{e}quations aux d\'{e}riv\'{e}es partielles,
  (2012).

\bibitem{FWwaves}
\leavevmode\vrule height 2pt depth -1.6pt width 23pt, {\em Wave packets in
  honeycomb structures and two-dimensional {D}irac equations}, Comm. Math.
  Phys., 326 (2014), pp.~251--286.

\bibitem{frankkonig}
{\sc R.~L. Frank and T.~K\"{o}nig}, {\em Classification of positive singular
  solutions to a nonlinear biharmonic equation with critical exponent}, Anal.
  PDE, 12 (2019), pp.~1101--1113.

\bibitem{diracoperators}
{\sc T.~Friedrich}, {\em Dirac operators in {R}iemannian geometry}, vol.~25 of
  Graduate Studies in Mathematics, American Mathematical Society, Providence,
  RI, 2000.
\newblock Translated from the 1997 German original by Andreas Nestke.

\bibitem{GrWu}
{\sc R.~E. Greene and H.~Wu}, {\em Integrals of subharmonic functions on
  manifolds of nonnegative curvature}, Invent. Math., 27 (1974), pp.~265--298.

\bibitem{nadineconformalinvariant}
{\sc N.~Grosse}, {\em On a conformal invariant of the {D}irac operator on
  noncompact manifolds}, Ann. Global Anal. Geom., 30 (2006), pp.~407--416.

\bibitem{nadineboundedgeometry}
\leavevmode\vrule height 2pt depth -1.6pt width 23pt, {\em Solutions of the
  equation of a spinorial {Y}amabe-type problem on manifolds of bounded
  geometry}, Comm. Partial Differential Equations, 37 (2012), pp.~58--76.

\bibitem{hijazi}
{\sc O.~Hijazi}, {\em A conformal lower bound for the smallest eigenvalue of
  the {D}irac operator and {K}illing spinors}, Comm. Math. Phys., 104 (1986),
  pp.~151--162.

\bibitem{isobecritical}
{\sc T.~Isobe}, {\em Nonlinear {D}irac equations with critical nonlinearities
  on compact {S}pin manifolds}, J. Funct. Anal., 260 (2011), pp.~253--307.

\bibitem{jannellisolimini}
{\sc E.~Jannelli and S.~Solimini}, {\em Concentration estimates for critical
  problems}, Ricerche Mat., 48 (1999), pp.~233--257.
\newblock Papers in memory of Ennio De Giorgi (Italian).

\bibitem{jost}
{\sc J.~Jost}, {\em Riemannian geometry and geometric analysis}, Universitext,
  Springer, Heidelberg, sixth~ed., 2011.

\bibitem{korevaar}
{\sc N.~Korevaar, R.~Mazzeo, F.~Pacard, and R.~Schoen}, {\em Refined
  asymptotics for constant scalar curvature metrics with isolated
  singularities}, Invent. Math., 135 (1999), pp.~233--272.

\bibitem{maalaoui}
{\sc A.~Maalaoui}, {\em Infinitely many solutions for the spinorial {Y}amabe
  problem on the round sphere}, NoDEA Nonlinear Differential Equations Appl.,
  23 (2016), pp.~Art. 25, 14.

\bibitem{raulot}
{\sc S.~Raulot}, {\em A {S}obolev-like inequality for the {D}irac operator}, J.
  Funct. Anal., 256 (2009), pp.~1588--1617.

\bibitem{schoen2}
{\sc R.~M. Schoen}, {\em Variational theory for the total scalar curvature
  functional for {R}iemannian metrics and related topics}, in Topics in
  calculus of variations ({M}ontecatini {T}erme, 1987), vol.~1365 of Lecture
  Notes in Math., Springer, Berlin, 1989, pp.~120--154.

\bibitem{struwedecomposition}
{\sc M.~Struwe}, {\em A global compactness result for elliptic boundary value
  problems involving limiting nonlinearities}, Math. Z., 187 (1984),
  pp.~511--517.

\bibitem{Teschl}
{\sc G.~Teschl}, {\em Ordinary differential equations and dynamical systems},
  vol.~140 of Graduate Studies in Mathematics, American Mathematical Society,
  Providence, RI, 2012.

\bibitem{remarkdirac}
{\sc C.~Wang}, {\em A remark on nonlinear {D}irac equations}, Proc. Amer. Math.
  Soc., 138 (2010), pp.~3753--3758.

\end{thebibliography}

\end{document}